\documentclass[11pt]{amsart}
\usepackage{amsmath,amssymb,amsthm,amsfonts,setspace,hyperref,dsfont, xcolor}

\newcommand{\NN}{\mathbb{N}}
\newcommand{\N}{\mathbb{N}}
\newcommand{\RR}{\mathbb{R}}
\newcommand{\R}{\mathbb{R}}

\newcommand{\C}{\mathbb{C}}

\newcommand{\Z}{\mathbb{Z}}

\newcommand{\norm}[1]{\lVert#1\rVert}

\renewcommand{\sl}{%
\operatorname{\mathfrak s\mathfrak l}
}
\newcommand \re {{%
\operatorname{Re}
}}
\newcommand{%
\vol} {{%
\operatorname{vol}
}}
 \DeclareMathOperator{\SL}{SL}
 
\newtheorem{theorem}{Theorem}[section]
\newtheorem {lemma} [theorem]{Lemma}
\newtheorem {proposition}[theorem]{Proposition}

\newtheorem{remark}[theorem]{Remark}

\title[Cohomological equation and cocycle rigidity ]
{Cohomological equation and cocycle rigidity of discrete parabolic actions}
\author{James Tanis}
\thanks{}   
\address[James Tanis]{The MITRE Corporation \\ McLean, VA 22102, USA \footnote{\textbf{Approved for Public Release; Distribution Unlimited. Case Number 18-1082.}  The first author's affiliation with The MITRE Corporation is provided for identification purposes only, and is not intended to convey or imply MITRE's concurrence with, or support for, the positions, opinions or viewpoints expressed by the authors. \textcopyright 2018 The MITRE Corporation.
All rights reserved}}
\email{jhtanis@mitre.org}

\author{ Zhenqi Jenny Wang$^1$ }
\thanks{ $^1$ Based on research supported by NSF grant   DMS-1700837}
\address[Zhenqi Jenny Wang]{Department of Mathematics\\
        Michigan State University\\
        East Lansing, MI 48824,   USA}
\email{wangzq@math.msu.edu}

\begin{document}
\begin{abstract}
We study the cohomological equation for discrete horocycle maps on $\SL(2,\RR)$ and $\SL(2,\RR)\times \SL(2,\RR)$ via representation theory.
Specifically, we prove Hilbert Sobolev non-tame estimates for solutions of the cohomological equation of horocycle maps in representations of $\SL(2,\RR)$.  Our estimates improve on previous results and are sharp up to a fixed, finite loss of regularity.  
Moreover, they are tame on a co-dimension one subspace of $\sl(2, \R)$, and we prove tame cocycle rigidity for some two-parameter discrete actions, improving on a previous result.  
Our estimates on the cohomological equation of horocycle maps overcome difficulties in previous papers by working in a more suitable model for $\SL(2, \R)$ in which all cases of irreducible, unitary representations of $\SL(2, \R)$ can be studied simultaneously.

Finally, our results combine with those of a very recent paper by the authors to give cohomology results for discrete parabolic actions in regular representations of some general classes of simple Lie groups, providing a fundamental step toward proving differential local rigidity of parabolic actions in this general setting.
\end{abstract}
\maketitle
\tableofcontents

\section{Introduction}

Cohomology arises in various problems in dynamical systems, such as those concerning the existence of invariant measures and mixing of suspension flows, and it is of central importance to rigidity and stability questions, see \cite{damjanovic2011local}, \cite{damianovic2005periodic}, \cite{damjanovic2010local} and \cite{katok1996cocycles}.

 There is an abundance of rigidity results for higher-rank (partially) hyperbolic algebraic systems. Local rigidity and cocycle rigidity for higher-rank standard hyperbolic algebraic actions were proved in the 1990's by Katok and Spatzier \cite{katok1994first}, \cite{katok1997differential}.  Then the results were extended to a large class of higher-rank partially hyperbolic actions, see for example \cite{damianovic2005periodic}, \cite{katok1996cocycles}, \cite{katok1994first}, \cite{katok1994subelliptic} and  \cite{Wang2017}. These results were proven using geometric arguments that rely on the fast separation of orbits.

Much less in known about parabolic systems.
Parabolic systems are fundamentally different from partially hyperbolic ones in that nearby orbits separate from each other at a much slower rate, namely, at so-called polynomial rather than exponential speed.  For this reason, the methods used to prove local or cocycle rigidity results for higher-rank partially hyperbolic systems do not work in the parabolic setting.  All known examples of parabolic systems are homogeneous, such as horocycle flows and Heisenberg nilflows.  So instead, a successful approach has been to make use of the underlying algebraic nature of the system and proceed via the theory of unitary representations.

Tame estimates for solutions to the cohomological equation of horocycle flows were obtained by Flaminio and Forni in \cite{flaminio2003invariant}.  Subsequently, several authors used these results to study the cohomological equation in some models of rank two continuous, parabolic actions on products of $\SL_2$ with coefficients in $\R$ or $\C$, see for example the works by Mieczkowski, Damjanovic-Katok and Ramirez in \cite{mieczkowski2007first}, \cite{damjanovic2011local} and \cite{ramirez2009cocycles}, respectively.
The second author of this paper studied rank two continuous parabolic actions on some general classes of simple, higher rank Lie groups in \cite{wang2015cohomological}.

There have also been %{\color{red} rank-one} 
cohomology results in the nilmanifold setting, which require a diophantine condition.   Flaminio and Forni proved tame estimates for the solution of the cohomoloigical equation of Heisenberg nilmanifolds in \cite{flaminio2006equidistribution}.  They later obtained non-tame estimates for solutions to the cohomological equation in higher step nilmanifolds in \cite{flaminio2014effective}.  Cosentino and Flaminio extended \cite{flaminio2006equidistribution} in a different direction, by increasing the dimension of the Heisenberg nilmanifold rather than the step.  They proved a tame splitting for the isotropic subgroups of the higher dimensional Heisenberg group.  
%{\color{red} Higher rank results??} 

Unlike in the partially hyperbolic setting, cohomology results for parabolic maps tend to be significantly harder to obtain than the corresponding results for parabolic flows.  For example, the space of obstructions is more complicated for maps than flows, as it tends to be infinite rather than finite dimensional in irreducible, unitary representations of the group.
Consequently, results for maps are more recent.  The first author proved non-tame estimates for solutions of the cohomological equation of horocycle maps in \cite{tanis2014cohomological}, which was improved to another non-tame estimate by the first author with Flaminio and Forni in \cite{flaminio2016effective}.  In \cite{damjanovic2014cocycle}, Damjanovic and Tanis used such estimates to prove a non-tame splitting for $\Z^2$ actions by horocycle maps on $\SL(2, \R) \times \SL(2, \R) / \Gamma$, where $\Gamma$ is an irreducible, cocompact lattice.
\smallskip

In this paper, we prove non-tame estimates for the solution to the cohomological equation of horocycle maps that are sharp up to a fixed finite loss of regularity, see Theorem~\ref{main_thm1}, thus improving on results in \cite{tanis2014cohomological} and \cite{flaminio2016effective}.
Estimates in previous papers were not sharp because with the unitary models that were used, the discrete series was different and harder than the other cases.
In this paper we were able to obtain sharp estimates by applying a suitable unitary transformation to the Fourier models of $\SL(2, \R)$ used in \cite{flaminio2016effective}, where all cases of irreducible representations (principal series, discrete series, etc. ) could be handled at once and in the same way.
We believe this version of the Fourier model will have other applications as well.  

A principal application for our sharp estimates
is toward the more general problem of obtaining sharp estimates (up to a fixed, finite loss of regularity)  for solutions to the cohomological equation of discrete parabolic maps in regular representations of some classes of possibly higher-rank, simple Lie groups, see Theorem~2.1 of \cite{tanis2017cohomological},
where analogous results were already obtained for parabolic flows in \cite{wang2015cohomological}.
A main reason that cohomological equations are difficult to study in the higher-rank setting is that the representation theory of higher rank simple Lie groups is very complicated.
So the crucial point is to study the problem in various subgroups,
such as $\SL(2, \R)$, where the representation theory is much simpler than that of the entire group.  Hence, our results on horocycle maps are used in \cite{tanis2017cohomological} to prove analogous results for regular representations of (some classes of) higher-rank simple Lie groups.

Moreover, even though the solution is not tame, we prove that estimates are tame in a co-dimension one subspace of the Lie algebra, see Theorem~\ref{eq:tame_X_V}, which is used in \cite{tanis2017cohomological} to deduce tame cocycle rigidity for abelian discrete parabolic actions in the higher-rank setting.  Specifically, tame estimates in a co-dimension one subspace allow us conclude that solutions to the cohomological equation of the discrete parabolic map in regular representations of (some classes of) higher-rank Lie groups are also tame in a co-dimension one subspace of the Lie algebra, and the tame cocycle rigidity result is deduced from this.  An important application of tame cocycle rigidity is to establish local rigidity by the KAM iterative scheme, see \cite{damjanovic2010local}, \cite{damjanovic2011local}.

To complete the picture, in this paper we also prove tame cocycle rigidity for discrete actions by horocycle maps on $\SL(2, \R) \times \SL(2, \R)$, improving on the non-tame cocycle rigidity estimates obtained in \cite{damjanovic2014cocycle}.  This may also be viewed as a motivating example for the analogous higher-rank result appearing in \cite{tanis2017cohomological}.

Finally, to keep the paper self-contained, we give a simpler proof of the non-tame lower bound for solutions to the cohomological equation of horocycle maps than the one appearing in \cite{tanis2017cohomological}, which concerned general root unipotent maps in $\SL(n ,\R)$.
To the best of our knowledge, this result, and the more general version in \cite{tanis2017cohomological}, are the first examples of non-tame solutions to a cohomological equation in the non-commutative, homogeneous setting.  

\subsection{Results}
The Lie algebra of $\SL(2, \R)$ is generated by the vector
fields
 \[
 X= \begin{pmatrix} {1}&0\\0& {-1}
 \end{pmatrix}, \quad U=\begin{pmatrix} 0&1\\0& 0
 \end{pmatrix}, \quad V=\begin{pmatrix} 0&0\\1& 0
 \end{pmatrix}.
 \]
which satisfy the commutation relations
\begin{equation}\label{eq:sl-com}
[U, V] = X\,, \ \ \ \ \  \ \ [X, U] = 2 U\,,\ \ \ \ \ \  \ \ [X, V] = -2V\,.
\end{equation}
Let $(\pi,\mathcal H)$ be a unitary representation of $\SL(2, \R)$ and let
\[
\triangle := -X^2 - 2(U^2 + V^2)
\]
be the Laplacian on $\mathcal{H}$.  The operator $(I + \triangle)$ is essentially self-adjoint on $\mathcal{H}$,
so for any $s \geq 0$, the spectral theorem gives that the
operator $(I + \triangle)^{s/2}$ is defined.
Then let $W^s(\mathcal{H}) \subset \mathcal{H}$ be the maximal domain of $(I + \triangle)^{s/2}$ on $\mathcal{H}$
equipped with the inner product
\[
\langle f, g \rangle_s = \langle (I + \triangle)^s f, g\rangle\,,
\]
where $\langle f, g \rangle$ is the $\mathcal{H}$ inner product.
For any $s \geq 0$, define
\begin{equation}\label{eq:f_norm_abstract}
\Vert f \Vert_s = \langle f, f \rangle^{1/2}_s\,.
\end{equation}
Let $W^\infty(\mathcal{H}) = \bigcap_{s > 0} W^s(\mathcal{H})$ be the space of smooth vectors in $\mathcal{H}$.

The above spaces have corresponding distributional dual spaces.
For any $s > 0$, let $W^{-s}(\mathcal{H})$ be the distributional dual space to $W^s(\mathcal{H})$,
and let $W^{-\infty}(\mathcal{H}) = \bigcup_{s > 0} W^{-s}(\mathcal{H})$ be
the distributional dual space to $W^\infty(\mathcal{H})$.

Next we present our results for regular representations of $\SL(2, \R)$ and $\SL(2, \R) \times \SL(2, \R)$.  We remark that they hold for general unitary representations $(\pi,\mathcal{H})$ of $\SL(2, \R)$ and $\SL(2, \R) \times \SL(2, \R)$ provided the relevant spectral gap property holds.  Namely, the spectral gap property for $\SL(2, \R)$ representations is described in Appendix~\ref{sect:SL2R_reps}, and
when $\mathcal H$ is an $\SL(2, \R) \times \SL(2, \R)$ representation, we require that the restriction of $\pi$ to any $\SL(2,\RR)$ factor has a spectral gap.

\subsubsection{One-parameter discrete actions}
%One of our primary results concerns the cohomological equation of horocycle maps. 

Let $\Gamma \subset \SL(2, \R)$ be a lattice and $M := \SL(2, \R) / \Gamma$.
Let $L^2(M)$ be the complex-valued, square-integrable functions on $M$ with respect to the $\SL(2, \R)$-invariant volume form.
Let $\pi$ be the regular representation of $\SL(2, \R)$ on $\mathcal H = L^2(M)$, and for $s \in \R\cup\{\pm\infty\}$, we use the notation
\[
W^{s}(M) := W^s(L^2(M))\,.
\]
The vector field $X$ generates the geodesic flow, and $U$ and $V$ generate the unstable and stable horocycle flows with respect to the geodesic flow, which all act by left multiplication on $M$.  In this paper, we refer to the stable horocycle flow as the horocycle flow $(h_t)_{t \in \R}$.  So for any $t \in \R$,
\[
h_t = \exp(t V)\,.
\]

For a given $L > 0$ and coboundary $g$, we find Sobolev estimates for the solution $f$
to the equation
\begin{equation}\label{eq:map-cohomological}
f \circ h_L - f:=\pi(h_L)f- f  = g\,.
\end{equation}
Invariant distributions for the horocycle map have already been classified in Theorem~1.1 of \cite{tanis2014cohomological},
and a more precise description of the regularity of these
distributions in irreducible representations of $\SL(2, \R)$
was given in Theorem~2.6 of \cite{flaminio2016effective}.
Denote the spaces of invariant distributions for $h_L$ in $W^{-s}(M)$ and $W^{-\infty}(M)$, respectively, by 
\[
\begin{aligned} 
& \mathcal I_L^s := \left\{ D \in W^{-s}(M) : D(f \circ h_L) = D(f) \text{ for any } f\in W^s(M)\right\}\,, \\ 
& \mathcal I_L := \left\{ D \in W^{-\infty}(M) : D(f \circ h_L) = D(f) \text{ for any } f\in W^\infty(M)\right\}\,.
\end{aligned}
\]
These invariant distributions were first classified in Theorem~1.2 of \cite{tanis2014cohomological}, 
and sharper estimates of their regularity were provided in the following theorem.  
\begin{theorem}[Theorem 2.6 of \cite{flaminio2016effective}]  
  \label{theo:invariant-maps} 
  Let 
  \[
  \mathcal I_{0}(M) := \left\{\mathcal D \in W^{-\infty}(M) :  V
  \mathcal D = \mathcal D\right\}\,,
\]
 and let $ \sigma_{\text{pp}} $ be the spectrum
  of the Laplace-Beltrami operator $ \triangle $ on $L^2(M)$.  Then in
  any Sobolev structure $ W^{-s}(M) $, for $ s > 0 $, there is a
  splitting
  \[
    \mathcal I_{L}(M) = \mathcal I_{0}(M) \oplus \mathcal I_{L,
    \text{twist}}(M)\,,
  \]
  where we have $ \mathcal I_{L, \text{twist}}(M) \subset W^{-(1/2+)}
  (M) $ and for each irreducible, unitary space $ H $, the space $
  \mathcal I_{L, \text{twist}}(\Gamma) \cap H^{-(1/2 +)} $ has infinite,
  countable dimension.

  The space $ \mathcal I_{0}(M) $ is described in Theorem 1.1 of
  \cite{flaminio2003invariant} as follows:  It has infinite, countable dimension.  It is a
  direct sum of the trivial representation $ \mathcal I_{%
  \vol} $ and irreducible, unitary representations $ \mathcal I_{\mu} $
  belonging to the principal series, the complementary series, the
  discrete series and the mock discrete series.

  Specifically,
  \begin{itemize}
    \item
      The space $ \mathcal I_{%
      \vol} $ is spanned by the $ \SL(2, \R) $-invariant volume;
    \item
      For $ 0 < \mu < 1 $, there is a splitting $ \mathcal I_\mu =
      \mathcal I_\mu^+ \oplus \mathcal I_\mu^- $, where $ \mathcal I_\mu^
      {\pm} \subset W^{-s}(M) $ if and only if $ s > \frac{1 \pm \sqrt{1
      - \mu}}{2} $, and each subspace has dimension equal to the
      multiplicity of $ \mu \in \text{spec}(\Box) $;
    \item
      If $ \mu \geq 1 $ and $ H_\mu $ is a principal series
      representation, then $ \mathcal I_\mu \subset W^{-s}(M) $ if and
      only if $ s > 1/2 $, and it has dimension equal to twice the
      multiplicity of $ \mu \in \text{spec}(\Box) $;
    \item
      If $ \mu = -n^2 + 2n $ for $ n \in \Z^+ $, then $ \mathcal I_\mu
      \subset W^{-s}(M) $ if and only if $ s > n/2 $ and it has
      dimension equal to twice the rank of the space of holomorphic
      sections of the $ n_{\text{th}} $ power of the canonical line
      bundle over $ M $.
  \end{itemize}
\end{theorem}

By Theorem~1.2 of \cite{tanis2014cohomological}, coboundaries for the map $h_L$
are those functions that are annihilated by the above invariant distributions.
So let
\[
\begin{aligned} 
& Ann_L^s := \left\{ f \in W^{s}(M) : D(f) = 0 \text{ for any } D\in \mathcal I_L\right\}\,, \\ 
& Ann_L := \left\{ f \in W^{\infty}(M) : D(f) = 0 \text{ for any } D\in \mathcal I_L\right\}\,.
\end{aligned}
\]

Our main theorem on solutions of the cohomological equation of horocycle maps
is the following improved bound.
Our approach follows \cite{flaminio2016effective} and \cite{tanis2017cohomological}
in that we first prove corresponding results for the twisted equation for horocycle flows,
\[
(V + \sqrt{-1} \lambda) f = g\,, \quad \lambda \in \R^*\,,
\]
and deduce the below result for horocycle maps from that.
As mentioned above, we improved on previous estimates by finding a more suitable version of the Fourier model used in \cite{tanis2017cohomological}, where we obtain the bound for all representation cases (principal, complementary, mock discrete and discrete series) simultaneously.

Let $G^V$ be the operator on the space of coboundaries for the horocycle flow given by
\begin{align}\label{for:2}
 V \ G^V(g) = g\,,
\end{align} 
which is uniquely defined (up to additive constants) on $W^s(M)$ when $s \geq 0$ and was studied in \cite{flaminio2003invariant}. 

\begin{theorem}\label{main_thm1}
For any $s \geq 1$, for any $\epsilon > 0$,
there is a constant $C_{s, \epsilon} > 0$ such that the following holds.
For any $L > 0$ and
for any $g \in Ann_L^{2s + 1 + \epsilon}(M)$, there is a solution
solution $f \in W^s(M)$ to the cohomological equation
\eqref{eq:map-cohomological}
such that
\begin{align}\label{for:3}
  \Vert f \circ h_{-L/2} \Vert_s \leq C_{s, \epsilon} \left(\frac{1 + L^{2s}}{L} \Vert G^{V}(g)\Vert_{s} + L^{\epsilon} \frac{(1 + L^s)}{\epsilon} \Vert g \Vert_{2s + 1 + \epsilon}\right)\,.
\end{align}
\end{theorem}
Moreover, the estimate is tame on a co-dimension one subspace of $\sl(2, \R)$.
\begin{theorem}\label{eq:tame_X_V}
Using the assumptions in the above theorem, we have
\[
\begin{aligned}
\Vert (I - X^2 - V^2)^{s/2} f \circ h_{-L/2} \Vert & \leq C_{s, \epsilon} \frac{1 + L^{2s}}{L} \Vert (I - X^2 - V^2)^{s/2} G_\mu^{V}(g)\Vert  \\
& + C_{s, \epsilon} L^{\epsilon} \frac{(1 + L^s)}{\epsilon} \Vert (I - X^2 - V^2)^{(s+1+\epsilon)/2} g \Vert\,.
\end{aligned}
\]
\end{theorem}
\begin{remark}
 %Theorem 1.1 of \cite{tanis2014cohomological} shows that if for any $g \in Ann_L^{2s+1+\epsilon}(\mathcal{H})$, the equation \eqref{for:2} has a solution $G^V(g)\in W^s(\mathcal{H})$; furthermore, f
Theorem 1.2 of \cite{flaminio2003invariant} shows that for any $s \geq 1$ and for any $\epsilon > 0$, 
\begin{align*}
 \norm{G^V(g)}_s\leq C_{s,\epsilon} \norm{g}_{s+1+\epsilon}.
\end{align*}
Then the above discussion and estimates \eqref{for:3} implies that
\begin{align}\label{for:6}
  \Vert f \circ h_{-L/2} \Vert_s \leq C_{s, \epsilon} \left(\frac{1 + L^{2s}}{L} \Vert g\Vert_{s+1+\epsilon} + L^{\epsilon} \frac{(1 + L^s)}{\epsilon} \Vert g \Vert_{2s + 1 + \epsilon}\right).
\end{align}
\end{remark}

It has already been shown from the lower bound in Theorem~2.2 of \cite{tanis2017cohomological} that the above non-tame estimate is sharp with respect to a loss of regularity of 1/2 derivatives.
However that proof was written for representations of $\SL(n, \R)$, for any $n \geq 2$,
and is overly complicated for the special case of $\SL(2, \R)$.
Hence, we provide a simpler and more transparent version of that proof in Section~\ref{sect:pf1}, which is specific to $\SL(2, \R)$.

\begin{theorem}\label{main_thm3}[special case of Theorem~2.2 of \cite{tanis2017cohomological}]
For any $s \geq 0$, for any $\sigma \in [0, s+1/2)$ and
for any $L > 0$, the following holds.
For every constant $C > 0$ there is a function
$g \in W^\infty(M)$ with a solution $f \in W^\infty(M)$
to the equation
\begin{equation}\label{eq:co_eqn-map}
f \circ h_L - f = g
\end{equation}
such that
\[
\Vert f \Vert_s > C \Vert g \Vert_{s +\sigma}\,.
\]
\end{theorem}

\subsubsection{Two-parameter discrete actions}\label{sec:2}
Let $\Gamma \subset \SL(2, \R) \times \SL(2, \R)$ be an irreducible lattice, and let $M = \SL(2, \R) \times \SL(2, \R) / \Gamma$.  Also let $L^2(M)$ and $W^s(M)$ be defined analogously as in the above $\SL(2, \R)$ case.
We consider the following unipotent maps acting by left multiplication on $M$:
\begin{equation}\label{eq:cocycle_maps}
h_{L_1}^{(1)}=\begin{pmatrix} {1}& L_1\\0& {1}
 \end{pmatrix}\times \begin{pmatrix} {1}&0\\0& {1}
 \end{pmatrix}\,, \quad h_{L_2}^{(2)}=\begin{pmatrix} {1}& 0\\0& {1}
 \end{pmatrix}\times \begin{pmatrix} {1}& L_2\\0& {1}
 \end{pmatrix}\,.
 \end{equation}
The next theorem shows that in contrast to the one-parameter setting, tame estimates hold for solutions to the cohomological equation of two-parameter actions. 
 %The next theorem shows that despite Theorem \ref{main_thm1}, where the estimates of the individual cohomological equation are not tame, estimates for cocycle equations are always tame.
\begin{theorem}\label{th:8}
Let $s \geq 0$ and suppose $f,\,g\in W^{s + 3}(M)$
 and satisfies the cocycle equation
 \begin{align*}
   f \circ h_{L_1}^{(1)} - f = g\circ h_{L_2}^{(2)} - g\,.  
 \end{align*}
Then there is a solution $p\in W^s(M)$
such that
\begin{align*}
  p\circ h_{L_1}^{(1)}-p=g\quad \text{and}\quad p\circ h_{L_2}^{(2)}-p=f
\end{align*}
with the Sobolev estimates
\begin{align}\label{for:5}
    \norm{p}_s\leq C_s(L+\frac{1}{L})\max\{\norm{f}_{s+3}, \,\norm{g}_{s+3}\},\qquad \forall\,s>0.
\end{align}
\end{theorem}

\subsection{Direct decompositions of Sobolev space}\label{sec:3}
For any Lie group $G$ of type $I$, such as $\SL(2, \R)$ and $\SL(2, \R)\times \SL(2, \R)$,
with a unitary representation $\rho$, there is a decomposition of $\rho$ into a direct integral
\begin{align}\label{for:66}
 \rho=\int_Z\rho_zd\mu(z)
\end{align}
of irreducible unitary representations for some measure space $(Z,\mu)$ (we refer to
\cite[Chapter 2.3]{zimmer2013ergodic} or \cite{margulis1982finitely} for more detailed account of direct integral theory). All the operators in the enveloping algebra are decomposable with respect to the direct integral decomposition \eqref{for:66}. Hence there exists for all $s\in\R$ an induced direct
decomposition of the Sobolev spaces:
\begin{align}\label{for:67}
\mathcal{H}^s=\int_Z\mathcal{H}_z^sd\mu(z)
\end{align}
with respect to the measure $d\mu(z)$.

The existence of the direct integral decompositions
\eqref{for:66}, \eqref{for:67} allows us to reduce our analysis of the
cohomological equation to irreducible unitary representations. This point of view is
essential for our purposes. See for example, Theorem 1.1 of \cite{damjanovic2014cocycle}.

\section{Proof of Theorems~\ref{main_thm1} and \ref{main_thm3}}\label{sect:pf1}
We begin with Theorem~\ref{main_thm1}, which has already been proven
for the principal, complementary and mock discrete
series in Theorems~3.4 and 3.19 of \cite{flaminio2016effective}.
Next we discuss the model that we use to prove Theorem~\ref{main_thm1},
where we can handle all representation cases together, including the discrete series.
The estimates follow the approach in the proof of Theorem~3.4 of \cite{flaminio2016effective}.

\subsection{A Fourier model for $\SL(2, \R)$}

The Hilbert space models for the principal, complementary, discrete and mock discrete series are discussed in Appendix~\ref{sect:SL2R_reps}.  We use the representation parameter
\[
\nu := \sqrt{1 - \mu}\,,
\]
for convenience, and we denote the real part of $\nu$ by $\re\nu$.

The Fourier transform for functions in the line model $H_\mu$ of the principal or complementary series is defined by
\[
\hat f(\xi) = \int_{\R} f(x) e^{- \sqrt{-1} \xi x} dx\,.
\]
Following \cite{flaminio2016effective}, for functions in the upper half-plane model of the discrete series or mock discrete series,
also denoted by $H_\mu$, the Fourier transform is given by
\[
\hat f(\xi) := \int_\R f(x + \sqrt{-1}) e^{-\sqrt{-1} \xi (x + \sqrt{-1})} dx\,.
\]
Cauchy's theorem implies that $\hat f$ is supported on $\R^+$ when $f \in H_\mu^\infty$ (see Lemma~3.15 of \cite{flaminio2016effective}).
By direct computation from the vector field formulas given in Appendix~\ref{sect:SL2R_reps},
we have $\hat V = - \sqrt{-1} \xi$ and
\begin{equation}\label{eq:Fourier_vf}
\hat X := (\nu - 1) - 2\xi \frac{d}{d\xi}\,, \quad \hat U := \sqrt{-1}\left((\nu-1) \frac{d}{d\xi} - \xi\frac{d^2}{d\xi^2}\right)\,.
\end{equation}
In all cases (i.e. principal, complementary, discrete and mock discrete series), Lemma~3.1 and Lemma~3.16 of \cite{flaminio2014effective} give a
constant $C_{\re\nu} > 0$ such that
\begin{equation}\label{eq:norm_f}
\Vert f \Vert^2 = C_{\re\nu}^2 \int_{\R} |\hat f(\xi)|^2 \vert \xi\vert^{-\re\nu} d\xi\,.
\end{equation}

In the discrete series, the number $\re\nu = \nu$ may be large.  The measure then $\xi^{-\re\nu} d\xi$ created difficulties in \cite{flaminio2016effective} and \cite{tanis2014cohomological} that led to weak estimates with respect to loss of regularity for the solution of the cohomological equation in the discrete series.
By an appropriate change of variable, the above norm is evaluated by an integral with the measure $\xi^{-1} d\xi$ for all irreducible representation cases, so we can handle all cases simultaneously and obtain sharp non-tame estimates.

To this end, fix any irreducible, unitary representation of the regular representation.
For any $J \subseteq \R$, define
\[
\mathbb L_\nu^2(J) := L^2(J\,, C_{\re\nu} \times \vert \xi\vert^{-1} d\xi)\,,
\]
and let $\mathcal A: H_\mu \to \mathbb L_\nu^2(\R)$ be the unitary transformation given by
\begin{equation}\label{eq:def_A}
\mathcal A  f(\xi) = \hat f(\xi) \xi^{-(\nu -1)/2}\,.
\end{equation}
Furthermore, for any $W \in \sl(2, \R)$, define
\[
\mathcal W = \mathcal A W \mathcal A^{-1}\,,
\]
so
\begin{equation}\label{eq:mathcal_vf}
\mathcal V = -\sqrt{-1} \xi\,, \quad \mathcal X = -2\xi \frac{d}{d\xi}\,, \quad \mathcal U = \sqrt{-1}\left(\frac{\nu^2 - 1}{4\xi}  - \xi\frac{d^2}{d\xi^2}\right)\,.
\end{equation}
One can verify that these formulas give the usual commutation relations
\begin{equation}\label{eq:commute_mathcal}
[\mathcal U\,, \mathcal V] = \mathcal X\,,\quad [\mathcal V\,, \mathcal X] = 2 \mathcal V\,, \quad [\mathcal U\,, \mathcal X] = -2 \mathcal U\,.
\end{equation}
In light of \eqref{eq:f_norm_abstract} and \eqref{eq:norm_f}, for any $s > 0$, we have
\begin{equation}\label{eq:s-norm-A}
\Vert f \Vert_s^2 = C_{\re\nu} \int_{\R} |(I - \mathcal V^2 - \mathcal X^2 - \mathcal U^2)^{s/2} \mathcal A f(\xi)|^2 \ \vert \xi\vert^{-1} d\xi\,.
\end{equation}

\subsection{Proof of Theorem~\ref{main_thm2}}
As in \cite{flaminio2016effective}, we derive estimates for the solution to the cohomological equation of horocycle maps from estimates of the solution $f$ to the twisted equation
\begin{equation}\label{eq:twist-cohomological}
(V + \sqrt{-1} \lambda) f = g\,,
\end{equation}
where $\lambda \in \R^*$.
The distributions that are invariant under the operator $(V + \sqrt{-1} \lambda)$
have already been studied in \cite{flaminio2016effective}.
For any irreducible, unitary representation $H_\mu$ of $\SL(2, \R)$, define
\[
\begin{aligned} 
& \mathcal I^{\lambda}_\mu:= \left\{ D \in H_\mu^{-\infty} : D((V + \sqrt{-1} \lambda) f) = 0 \text{ for any } f\in H_\mu^\infty \right\}\,, \\ 
& Ann^{\lambda}_\mu := \left\{ f \in H_\mu^\infty : D(f) = 0 \text{ for any } D\in \mathcal I^\lambda_\mu \right\}\,.
\end{aligned}
\]
Lemmas~3.3 and 3.14 of \cite{flaminio2016effective} show that
for any $\epsilon > 0$, $\mathcal I^{\lambda}_\mu \subset H_\mu^{-(1/2+\epsilon)}$.  
For $s > 1/2$, define 
\[
Ann_\mu^{\lambda, s} := \left\{ f \in H_\mu^{s} : D(f) = 0 \text{ for any } D\in \mathcal I^{\lambda}_\mu \right\}\,.
\] 

We derive sharp bounds in Theorem~\ref{main_thm1} from the following sharp abounds on
solutions to the twisted equation \eqref{eq:twist-cohomological}.  The strategy of deriving bounds for maps for from those of the twisted equation was already used in \cite{flaminio2016effective}.
\begin{theorem}\label{main_thm2} 
For any $s \geq 0$, for any $\epsilon$,
there is a constant $C_{s, \epsilon} > 0$ such that the following holds.  
For any irreducible, unitary representation $H_\mu$ of $\SL(2, \R)$.  
For any $\lambda \in \R^*$ and
for any $g \in Ann^{\lambda, 2s+1}_\mu$, there is a solution
solution $f \in H^s_\mu$ to the twisted equation \eqref{eq:twist-cohomological}
such that
\[
\Vert f \Vert_s \leq C_{s, \epsilon} \frac{(1 + \vert \lambda \vert^{-s})}{\vert \lambda \vert} \Vert g \Vert_{2s + 1} \,.
\]
As above, the estimate is tame when the norm involves only the $X$ and $V$ derivatives,
\[
\begin{aligned}
\Vert (I - X^2 - V^2)^{s/2} f \Vert \leq \frac{C_{s}}{\vert \lambda \vert} \Vert (I - X^2 - V^2)^{(s + 1)/2} g \Vert \,.
\end{aligned}
\]
\end{theorem}
First let $s \in \N$.
By the triangle inequality, the commutation relations~\eqref{eq:sl-com}
and formula~\eqref{eq:s-norm-A},
there is a constant $C_s > 0$ such that
\begin{align}
\Vert f \Vert_{2s} & \leq C_s \sum_{\alpha + \beta + \gamma \leq 2s} \Vert U^\alpha X^\beta V^\gamma f \Vert  \notag \\
& = C_s \sum_{\alpha + \beta + \gamma \leq 2s} \Vert \mathcal U^\alpha \mathcal X^\beta \mathcal V^\gamma \mathcal A f \Vert_{\mathbb L_\nu^2(\R)}\,. \label{eq:f_est:0}
\end{align}

Because $\mathcal A$ is unitary, estimates of solutions to the twisted equation \eqref{eq:twist-cohomological} are equivalently
obtained in $\mathbb L_\nu^2(\R)$ by solving
\[
(\mathcal V + \sqrt{-1} \lambda) \mathcal A f = \mathcal A g
\]
using the vector fields \eqref{eq:mathcal_vf}.
It follows that
\begin{equation}\label{eq:Af_def}
\mathcal A f(\xi) = \sqrt{-1} \frac{\mathcal A g(\xi)}{\xi - \lambda}\,.
\end{equation}

Next, as in \cite{flaminio2016effective}, 
set $(\mathcal A f)_\lambda(\xi) = (\mathcal A f)(\lambda \xi)$ for any function $f$, 
which means
\[
(\mathcal A f)_\lambda(\xi) = \frac{\sqrt{-1}}{\lambda} \frac{(\mathcal A g)_\lambda(\xi)}{\xi - 1}\,.
\]
To simplify notation until the end of the argument, for any function $f$, set
\[
f = (\mathcal A f)_\lambda\,,
\]
so \eqref{eq:Af_def} is reduced to solving
\begin{equation}\label{eq:twist_est}
f(\xi) = \frac{\sqrt{-1}}{\lambda} \frac{g(\xi)}{\xi - 1}\,.
\end{equation}

Notice that for any $\sigma \in \N$, we have
\[
(\mathcal V + \sqrt{-1} \lambda) (\mathcal V^\sigma f) = \mathcal V^\sigma g\,,
\]
so
\begin{equation}\label{eq:V-commute:1}
\mathcal V^\sigma f(\xi) = \sqrt{-1} \frac{\mathcal V^\sigma g(\xi)}{\xi - \lambda}\,.
\end{equation}
Hence, we can obtain an estimate of \eqref{eq:f_est:0} by an estimate involving only the $\mathcal U$ and $\mathcal X$ derivatives.
In light of \eqref{eq:twist_est}, let $I = [\frac{1}{2}, \frac{3}{2}]$, so
\[
\Vert \mathcal U^\beta \mathcal X^{\alpha} f \Vert_{\mathbb L_\nu^2(\R)} \leq \Vert \mathcal U^\beta \mathcal X^{\alpha} f \Vert_{\mathbb L_\nu^2(\R\setminus I)} + \Vert \mathcal U^\beta \mathcal X^{\alpha} f \Vert_{\mathbb L_\nu^2(I)}\,.
\]
The next two propositions estimate each term of the above inequality.  
The below proposition does not use a distributional assumption on $g$.  

\begin{proposition}\label{prop:no_lambda}
For any $\alpha, \beta \in \N$,
There is a constant $C_{\alpha, \beta}^{(0)} > 0$ such that
\[
\Vert \mathcal U^\beta \mathcal X^{\alpha} f \Vert_{\mathbb L_\nu^2(\R\setminus I)} \leq \frac{C_{\alpha, \beta}^{(0)}}{\vert \lambda \vert} \sum_{\substack{j + k \leq \alpha +  \beta \\ j \leq \beta}} \Vert \mathcal U^j \mathcal X^k g \Vert_{\mathbb L_\nu^2(\R)}\,.
\]
\end{proposition}
\begin{proof}
We start with a technical lemma.
\begin{lemma}\label{lemm:XU-f:1}
Let $\alpha, \beta \in \N$.
There are coefficients $(b_{j k l m n}^{(\alpha \beta)}) \subset \Z$
such that for any $g_1, g_2 \in H_\mu^\infty$,
\[
\mathcal U^\beta \mathcal X^\alpha (g_1 g_2) = i^\beta \sum_{\substack{j + k + l +n \leq \alpha +  \beta \\ j +l+ m \leq \beta}}  b_{j k l m n}^{(\alpha, \beta)} \left((\mathcal U^j \mathcal X^{k} g_1) ((\xi \frac{d^2}{d\xi^2})^l (\frac{d}{d\xi})^m \mathcal X^{n} g_2)\right)\,.
\]
\end{lemma}
\begin{proof}
By the product rule for ordinary differentiation, we get
\[
\begin{aligned}
\mathcal X (g_1 g_2)
& = (\mathcal X g_1) g_2 + g_1 (\mathcal X g_2)\,.
\end{aligned}
\]
Hence, for any $\alpha \in \N$, there are
coefficients $(a_{j k}^{(\alpha)} \subset \N$ such that
\begin{equation}\label{eq:X_alpha}
\mathcal X^\alpha (g_1 g_2) = \sum_{r + k = \alpha} a_{j k}^{(\alpha)} (\mathcal X^j g_1) (\mathcal X^k g_2)\,.
\end{equation}
Next, a short computation gives
\[
\begin{aligned}
\mathcal U (g_1 g_2)
& = \sqrt{-1} \left((\mathcal U g_1) g_2 + (\mathcal X g_1) (\frac{d}{d\xi} g_2) - g_1 (\xi \frac{d^2}{d\xi^2} g_2)\right)\,.
\end{aligned}
\]
Because
\[
[\xi \frac{d^2}{d\xi^2}\,, \frac{d}{d\xi}] = - \frac{d^2}{d\xi^2}\,,
\]
we get by induction that
\begin{equation}\label{eq:Leibnitz-U}
\mathcal U^\beta (g_1 g_2) = i^\beta \sum_{\substack{j + k + l \leq \beta \\ j+ l + m \leq \beta}}
b_{j k l m }^{(\beta)} (\mathcal U^j \mathcal X^k g_1) ((\xi \frac{d^2}{d\xi^2})^l (\frac{d}{d\xi})^m g_2)\,.
\end{equation}

By the above equality and \eqref{eq:X_alpha},
we get constants $(c_{j k l m n}^{(\alpha, \beta)}) \subset \Z$ such that
\[
\begin{aligned}
\mathcal U^\beta & \mathcal X^\alpha (g_1, g_2) = \sum_{n_1 + n_2 = \alpha} a_{n_1 n_2}^{(\alpha)} \mathcal U^\beta ((\mathcal X^{n_1} g_1) (\mathcal X^{n_2} g_2)) \\
& = i^\beta \sum_{n_1 + n_2 = \alpha} a_{n_1 n_2}^{(\alpha)} \sum_{\substack{j + k + l \leq \beta \\ j+ l + m\leq \beta}}
b_{j k l m}^{(\beta)} (\mathcal U^j \mathcal X^{k+n_1} g_1) ((\xi \frac{d^2}{d\xi^2})^l (\frac{d}{d\xi})^m \mathcal X^{n_2} g_2)  \\
& = i^\beta \sum_{\substack{j + k + l +n \leq \alpha +  \beta \\ j + l + m \leq \beta}}  c_{j k l m n}^{(\alpha, \beta)} (\mathcal U^j \mathcal X^{k} g_1) \left((\xi \frac{d^2}{d\xi^2})^l (\frac{d}{d\xi})^m \mathcal X^{n} g_2\right)\,.
\end{aligned}
\]
\end{proof}

We now finish the proof of Proposition~\ref{prop:no_lambda}.
In the above lemma, let $g = g_1$ and $g_2 = (\xi - 1)^{-1}$.
An induction argument shows there are coefficients $(c_{j, k}^{(l, m, n)})$ such that
\begin{equation}\label{eq:decomp_constant_fact}
(\xi \frac{d^2}{d\xi^2})^l (\frac{d}{d\xi})^m \mathcal X^{n} = \sum_{\substack{j \leq l + n \\ k \leq 2l + m + n \\ k - j \geq l+m}} c_{j, k}^{(l, m, n)} \xi^j (\frac{d}{d\xi})^{k}\,.
\end{equation}
Then from the definition of $\mathcal X$, and because
\[
(\frac{d}{d\xi})^j (\xi - 1)^{-1} = (-1) j! (\xi - 1)^{-(j +1)}\,,
\]
we get a constant $C_{l, m, n} > 0$ such that for any $\xi \in \R \setminus I$,
\[
\left|(\xi \frac{d^2}{d\xi^2})^l (\frac{d}{d\xi})^m \mathcal X^{n} \left(\frac{1}{\xi - 1}\right)\right| \leq C_{l, m, n} \frac{1}{\vert \xi - 1 \vert}\,.
\]
By the triangle inequality and Lemma~\ref{lemm:XU-f:1}, it follows that there is a constant $C_{\alpha, \beta} > 0$ such that
\begin{align}
\Vert & \mathcal U^\beta \mathcal X^\alpha f \Vert_{\mathbb L_\nu^2(\R\setminus I)} = \frac{1}{\vert \lambda \vert} \Vert \mathcal U^\beta \mathcal X^\alpha \frac{g(\xi)}{\xi - 1} \Vert_{\mathbb L_\nu^2(\R\setminus I)} \notag \\
& \leq \frac{C_{\alpha, \beta}}{\vert \lambda \vert} \sum_{\substack{j + k + l +n \leq \alpha +  \beta \\ j +l+ m \leq \beta}}  \vert c_{j k l m n}^{(\alpha, \beta)}\vert \Vert (\mathcal U^j \mathcal X^{k} g) \left((\xi \frac{d^2}{d\xi^2})^l (\frac{d}{d\xi})^m \mathcal X^{n} \frac{1}{\xi -1}\right) \Vert_{\mathbb L_\nu^2(\R\setminus I)} \notag \\
& \leq \frac{C_{\alpha, \beta}}{\vert \lambda \vert} \sum_{\substack{j + k \leq \alpha +  \beta \\ j \leq \beta}}  \Vert \frac{\mathcal U^j \mathcal X^{k} g}{\xi-1} \Vert_{\mathbb L_\nu^2(\R\setminus I)} \notag \\
& \leq \frac{C_{\alpha, \beta}}{\vert \lambda \vert} \sum_{\substack{j + k \leq \alpha +  \beta \\ j \leq \beta}}  \Vert \mathcal U^j \mathcal X^{k} g\Vert_{\mathbb L_\nu^2(\R\setminus I)}\,. \label{eq:twist_away1}
\end{align}
The proof of Proposition~\ref{prop:no_lambda} is now complete.
\end{proof}

In the next proposition, we use the distributional assumption to estimate $f$ over the interval $I$.      
Specifically, it is shown in \cite{flaminio2016effective} that the space $\mathcal I_\mu^1$ is one-dimensional and spanned by the functional 
\begin{equation}\label{eq:inv_dist_twist} 
\mathcal D(f) = f(1)\,.
\end{equation} 
From 
\[
\frac{d}{d\xi} = \frac{1}{-2\xi} \mathcal X 
\]
and the Sobolev embedding theorem, 
it follows that $\mathcal D \in H_\mu^{-(1/2+\epsilon)}$ for any $\epsilon > 0$.  
See Section 3.2.1 of \cite{flaminio2016effective} for details.  

\begin{proposition}\label{prop:w_lambda}   
For every $\alpha, \beta \in \N$, there is a constant $C_{\alpha, \beta}^{(1)} > 0$ such that 
for any $g \in Ann_\mu^{1, \alpha + 2\beta + 1}$, 
\[
\Vert \mathcal U^\beta \mathcal X^{\alpha} f \Vert_{\mathbb L_\nu^2(I)} \leq \frac{C_{\alpha, \beta}^{(1)}}{\vert \lambda \vert} \sum_{j + k \leq \alpha + 2\beta + 1} \Vert \mathcal U^j \mathcal X^k g \Vert_{\mathbb L_\nu^2(\R)}\,.
\]
\end{proposition}
\begin{proof}
First let
\[
\xi_t = t(\xi - 1) + 1\,. 
\] 
Notice that $g \in Ann_\mu^{1, \alpha + 2\beta + 1}$ implies $g(1) = 0$.    
Then as in Lemma~3.5 of \cite{flaminio2016effective},  
by formula \eqref{eq:twist_est} and the fundamental theorem of calculus we get that for any $\xi > 0$, 
\[ 
f(\xi) = \frac{\sqrt{-1}}{\lambda} \int_0^1 g'(\xi_t) dt \,.
\]

We now obtain a formula for $\mathcal U^\beta \mathcal X^{\alpha} f$, where $\alpha + \beta > 0$.

\begin{lemma}\label{lemm:XUg}
For any $t \in [0, 1]$, let
\[
\begin{aligned}
& W_0(t) := \xi \mathcal U  - \sqrt{-1} \frac{(t-1)^2}{4 \xi^2} (\mathcal X^2 + 2 \chi)\,, \\
& W_1(t) := (1 - \frac{1 - t}{\xi}) \mathcal X\,.
\end{aligned}
\]
For any $\beta \in \N\setminus\{0\}$, let
\[
\mathcal L^{(\beta)} :=
\left\{
\begin{aligned}
& \{(0, 0, 0)\} & \textrm{ if } \beta = 0\,, \\
& \left\{\mathbf l = (l_0, l_1, l_2) \in \N^2 : l_0 + l_1 + l_3 = \beta \,, l_0 \geq 1\right\} & \textrm{otherwise}\,.
\end{aligned}
\right.
\]
and for $\mathbf l = (l_0, l_1, l_2) \in \mathcal L^{(\beta)}$,
let $\mathcal L_{\mathbf l}^{(\beta)}$ be the set of all sequences of length $l_0 + l_1+l_3$ consisting of $l_0$ 0's, $l_1$ 1's and $l_2$ 2's.

Then for any $\alpha, \beta \in \N$ and for any $g \in Ann_\mu^{1, \alpha + 2\beta + 1}$, we have
\[
\begin{aligned}
\mathcal U^\beta \mathcal X^{\alpha}  f(\xi) & = \frac{\sqrt{-1}}{\lambda} \xi^{-\beta} \sum_{\substack{\mathbf l \in \mathcal L^{(\beta)} \\ (s_k) \in \mathcal L_{\mathbf l}^{(\beta)}}} c_{(s_k)} \int_0^1 [\prod_{k = 1}^{\beta} W_{s_k}(t) \left((1 - \frac{1 - t}{\xi}) \mathcal X\right)^\alpha   g'](\xi_t) dt\,.
\end{aligned}
\]
\end{lemma}
\begin{proof}
Hence, as in formula~(48) of \cite{flaminio2016effective}, we have
\begin{align}
\mathcal X  f(\xi)
& = \frac{\sqrt{-1}}{\lambda} \int_0^1 (-2\xi \frac{d}{d\xi})  g'(\xi_t) dt \notag \\
& = \frac{\sqrt{-1}}{\lambda} \int_0^1 -2[(\xi - (1 - t)) \frac{d}{d\xi}  g'](\xi_t) dt \notag \\
& = \frac{\sqrt{-1}}{\lambda} \int_0^1 [(-2\xi\frac{d}{d\xi}) + 2 (1 - t)\frac{d}{d\xi})  g'](\xi_t) dt \notag \\
& = \frac{\sqrt{-1}}{\lambda} \int_0^1 [(\mathcal X - \frac{1 - t}{\xi}(-2 \xi \frac{d}{d\xi})  g'](\xi_t) dt \notag \\
& = \frac{\sqrt{-1}}{\lambda} \int_0^1 [(1 - \frac{1 - t}{\xi}) \mathcal X  g'](\xi_t) dt \label{eq:X_g:1}
\end{align}
Hence, for any $\alpha \in \N$, we have
\begin{equation}\label{eq:X_g:2}
\eqref{eq:X_g:1} = \frac{\sqrt{-1}}{\lambda} \int_0^1 [\left((1 - \frac{1 - t}{\xi}) \mathcal X\right)^\alpha  g'](\xi_t) dt\,.
\end{equation}

Next, we have
\begin{equation}\label{eq:U-g}
\mathcal U  f(\xi) = \frac{\sqrt{-1}}{\lambda} \int_0^1 \sqrt{-1} \left(\frac{\nu^2 - 1}{4\xi} - \xi \frac{d^2}{d\xi^2}\right)  g'(\xi_t)\,.
\end{equation}
Using $\xi t = \xi_t - (1 - t)$, we get
\[
\frac{\nu^2 - 1}{4\xi} = t \left(\frac{\nu^2 - 1}{4 \xi_t} - \frac{(\nu^2 - 1)(t-1)}{4\xi_t(\xi_t + t-1)}\right)\,,
\]
and
\[
\xi \frac{d^2}{d\xi^2}  g'(\xi_t) = t\left( [\xi\frac{d^2}{d\xi^2}  g'](\xi_t) + (t - 1)[\frac{d^2}{d\xi^2}  g'](\xi_t)\right)\,.
\]
Then by the above two equalities, we have
\begin{align}
\eqref{eq:U-g} & = \frac{\sqrt{-1}}{\lambda} \int_0^1 t [\sqrt{-1}(\frac{\nu^2 - 1}{4\xi} - \xi \frac{d^2}{d\xi^2})  g'](\xi_t) dt \notag \\
& - \frac{\sqrt{-1}}{\lambda} \int_0^1 i t [\frac{t - 1}{\xi + t - 1} \left(\frac{\nu^2 - 1}{4\xi} + (\xi + t - 1) \frac{d^2}{d\xi^2}\right)  g'](\xi_t) dt \notag \\
& = \frac{\sqrt{-1}}{\lambda} \int_0^1 t [\mathcal U  g'](\xi_t) - t [\frac{t - 1}{\xi + t - 1} \left(\mathcal U + \sqrt{-1} (t - 1) \frac{d^2}{d\xi^2}\right)  g'](\xi_t) dt \notag \\
& = \frac{\sqrt{-1}}{\lambda} \int_0^1 t [\frac{\xi}{\xi + t-1} \mathcal U  - \sqrt{-1} \frac{(t-1)^2}{\xi+t+1} \frac{d^2}{d\xi^2}  g'](\xi_t) dt \notag \\
& = \frac{\sqrt{-1}}{\lambda} \int_0^1 [\frac{t}{\xi + t-1} \left(\xi \mathcal U  - \sqrt{-1} (t-1)^2 \frac{d^2}{d\xi^2}\right)  g'](\xi_t) dt\,. \label{eq:Ug:2}
\end{align}
Observe
\[
\frac{d^2}{d\xi^2} = \frac{1}{4 \xi^2} (\mathcal X^2 + 2 \chi)\,,
\]
so we conclude
\begin{align}
\eqref{eq:Ug:2} & = \frac{\sqrt{-1}}{\lambda} \int_0^1 [\frac{t}{\xi + t-1} \left(\xi \mathcal U  - \sqrt{-1} \frac{(t-1)^2}{4 \xi^2} (\mathcal X^2 + 2 \chi) \right)  g'](\xi_t) dt \notag \\
& = \frac{\sqrt{-1}}{\lambda} \int_0^1 \frac{1}{\xi} [\left(\xi \mathcal U  - \sqrt{-1} \frac{(t-1)^2}{4 \xi^2} (\mathcal X^2 + 2 \chi) \right)  g'](\xi_t) dt  \,. \label{eq:Uf_final}
\end{align}

Now we compute the formula for $\mathcal U^\beta f(\xi)$.
Observe that for any $m\in \Z$ and for any $\xi \in I$, we have
\begin{equation}\label{eq:commutation-xi}
\begin{aligned}
& [\mathcal U\,, \xi^m]  = \sqrt{-1} \left(m \xi^{m-1} \mathcal X  - m(m-1) \xi^{m-1} \right)\,, \\
& [\mathcal X\,, \xi^m]  = -2m \xi^{m} \,.
\end{aligned}
\end{equation}
Then for $W_0, W_1$ as in the statement of the lemma,
by an induction argument using the above commutation relations and by \eqref{eq:Uf_final},
we get
\[
\mathcal U^\beta  f(\xi) = \frac{\sqrt{-1}}{\lambda \xi^{\beta}}  \sum_{\substack{\mathbf l \in \mathcal L^{(\beta)} \\ (s_k) \in \mathcal L_{\mathbf l}^{(\beta)}}} c_{(s_k)} \int_0^1 [\prod_{k = 1}^{\beta} W_{s_k}(t)  g'](\xi_t) dt\,,
\]
where $c_{(s_k)} \in \C$ is a constant depending on the sequence $(s_k) \in \mathcal L_{\mathbf l}^{(\beta)}$.

The lemma follows from the above equality and formula \eqref{eq:X_g:2}.
\end{proof}

\[
\begin{aligned}
\mathcal U^\beta \mathcal X^{\alpha}  f(\xi) & = \frac{\sqrt{-1}}{\lambda} \xi^{-\beta} \sum_{\substack{\mathbf l \in \mathcal L^{(\beta)} \\ (s_k) \in \mathcal L_{\mathbf l}^{(\beta)}}} c_{(s_k)} \int_0^1 [\prod_{k = 1}^{\beta} W_{s_k}(t) \left((1 - \frac{1 - t}{\xi}) \mathcal X\right)^\alpha  g'](\xi_t) dt\,.
\end{aligned}
\]
By the above lemma and Minkowski's inequality, and because the interval $I$ is bounded away from zero, we get a constant $C_\beta > 0$ such that
\begin{align}
\Vert & \mathcal U^\beta \mathcal X^{\alpha}  f \Vert_{\mathbb L_\nu^2(I)} \notag \\
& \leq \frac{C_\beta}{\vert\lambda\vert}  \sum_{\substack{\mathbf l \in \mathcal L^{(\beta)} \\ (s_k) \in \mathcal L_{\mathbf l}^{(\beta)}}} \Vert  \int_0^1 [\prod_{k = 1}^{\beta} W_{s_k}(t) \left((1 - \frac{1 - t}{\xi}) \mathcal X\right)^\alpha  g'](\xi_t) dt\Vert_{\mathbb L_\nu^2(I)} \notag \\
& \leq \frac{C_\beta}{\vert\lambda\vert}  \sum_{\substack{\mathbf l \in \mathcal L^{(\beta)} \\ (s_k) \in \mathcal L_{\mathbf l}^{(\beta)}}}  \int_0^1 \Vert  [\prod_{k = 1}^{\beta} W_{s_k}(t) \left((1 - \frac{1 - t}{\xi}) \mathcal X\right)^\alpha  g'](t(\xi - 1) + 1)\Vert_{\mathbb L_\nu^2(I)} dt\,. \label{eq:f_I:4}
\end{align}
Because $\frac{d}{d\xi} = \frac{-1}{2\xi} \mathcal X$, it follows that
\begin{align}
\eqref{eq:f_I:4} & = \frac{C_\beta}{\vert\lambda\vert}  \sum_{\substack{\mathbf l \in \mathcal L^{(\beta)} \\ (s_k) \in \mathcal L_{\mathbf l}^{(\beta)}}}  \int_0^1 \notag \\
&\times \Vert  [\prod_{k = 1}^{\beta} W_{s_k}(t) \left((1 - \frac{1 - t}{\xi}) \mathcal X\right)^\alpha  \frac{1}{2\xi}\mathcal X g](t(\xi - 1) + 1)\Vert_{\mathbb L_\nu^2(I)} dt \,. \label{eq:f_I:5}
\end{align}
Notice that for any $\xi \in I$ and any $t \in [0, 1]$, there is a constant $C > 0$, depending only on $I$, such that
\[
\frac{1}{\xi_t} > C\,.
\]
Then using the commutation relations from \eqref{eq:commutation-xi} again and expanding the expression, we get
a constant $C_{\alpha, \beta} > 0$ such that
\[
\begin{aligned}
\eqref{eq:f_I:5} & \leq \frac{C_{\alpha, \beta}}{\vert \lambda\vert} \int_0^1 \sum_{\substack{j + k \leq \alpha + 2\beta + 1\\ j \leq \beta}} \Vert \mathcal U^j \mathcal X^k g \Vert_{\mathbb L_\nu^2(\R)} dt  \\
& \leq \frac{C_{\alpha, \beta}}{\vert \lambda\vert} \sum_{\substack{j + k \leq \alpha + 2\beta + 1\\ j \leq \beta}} \Vert \mathcal U^j \mathcal X^k g \Vert_{\mathbb L_\nu^2(\R)}\,.
\end{aligned}
\]
\end{proof}

\begin{proof}[Proof of Theorem~\ref{main_thm2}]
By Propositions~\ref{prop:no_lambda} and \ref{prop:w_lambda},
we get a constant $C_{\alpha, \beta} > 0$ such that
\begin{align}
\Vert \mathcal U^\beta & \mathcal X^\alpha f \Vert_{\mathbb L_\nu^2(\R)} \leq \Vert \mathcal U^\beta \mathcal X^\alpha f \Vert_{\mathbb L_\nu^2(\R\setminus I)} + \Vert \mathcal U^\beta \mathcal X^\alpha f \Vert_{\mathbb L_\nu^2(I)} \notag \\
& \leq \frac{C_{\alpha, \beta}}{\vert \lambda \vert} \left(\sum_{\substack{j + k \leq \alpha +  \beta \\ j \leq \beta}} \Vert \mathcal U^j \mathcal X^k g \Vert_{\mathbb L_\nu^2(\R)} + \sum_{\substack{j + k \leq \alpha + 2\beta + 1\\ j \leq \beta}} \Vert \mathcal U^j \mathcal X^k g \Vert_{\mathbb L_\nu^2(\R)} \right) \notag \\
& \leq \frac{C_{\alpha, \beta}}{\vert \lambda \vert} \sum_{\substack{j + k \leq \alpha + 2\beta + 1\\ j \leq \beta}} \Vert \mathcal U^j \mathcal X^k g \Vert_{\mathbb L_\nu^2(\R)}\,. \label{eq:twist_final:1}
\end{align}

The above gives the estimate in Theorem~\ref{main_thm2} for functions $(\mathcal A f)_\lambda$,
where $(\mathcal A f)_\lambda(\xi) =  (\mathcal A f)(\lambda \xi)$.
For the general result, we argue as in formula~(58) of \cite{flaminio2016effective}.
Observe that on smooth functions $f$, we have
\[
\mathcal U^\beta \mathcal X^\alpha  f_\lambda = \lambda^\beta (\mathcal U^\beta \mathcal X^\alpha  f)_\lambda\,.
\]
Hence,
\begin{align}
\Vert \mathcal U^\beta \mathcal X^\alpha  f \Vert_{\mathbb L_\nu^2( \R)} & = \vert \lambda\vert^{-\beta} \Vert (\mathcal U^\beta \mathcal X^\alpha  f_\lambda)_{1/\lambda} \Vert_{\mathbb L_\nu^2( \R)} \notag \\
& = \vert \lambda\vert^{-\beta + 1/2} \Vert \mathcal U^\beta \mathcal X^\alpha  f_\lambda \Vert_{\mathbb L_\nu^2( \R)} \notag \\
& \leq \frac{C_{\alpha, \beta}}{\vert \lambda \vert} \vert \lambda\vert^{-\beta + 1/2} \sum_{\substack{j + k \leq \alpha + 2\beta + 1\\ j \leq \beta}} \Vert \mathcal U^j \mathcal X^k  g_\lambda \Vert_{\mathbb L_\nu^2( \R)} \notag \\
& \leq \frac{C_{\alpha, \beta}}{\vert \lambda \vert} \vert \lambda\vert^{-\beta + 1/2} \sum_{\substack{j + k \leq \alpha + 2\beta + 1\\ j \leq \beta}} \vert \lambda \vert^{j} \Vert (\mathcal U^j \mathcal X^k  g)_\lambda \Vert_{\mathbb L_\nu^2( \R)} \notag \\
& \leq \frac{C_{\alpha, \beta}}{\vert \lambda \vert}  \sum_{\substack{j + k \leq \alpha + 2\beta + 1\\ j \leq \beta}} \vert \lambda \vert^{j - \beta + 1/2} \Vert (\mathcal U^j \mathcal X^k  g)_\lambda \Vert_{\mathbb L_\nu^2( \R)} \notag \\
& = \frac{C_{\alpha, \beta}}{\vert \lambda \vert}  \sum_{\substack{j + k \leq \alpha + 2\beta + 1\\ j \leq \beta}} \vert \lambda \vert^{j - \beta} \Vert \mathcal U^j \mathcal X^k  g \Vert_{\mathbb L_\nu^2( \R)} \notag \\
& \leq \frac{C_{\alpha, \beta}}{\vert \lambda \vert} (1 + \vert \lambda \vert^{-\beta}) \sum_{\substack{j + k \leq \alpha + 2\beta + 1\\ j \leq \beta}} \Vert \mathcal U^j \mathcal X^k  g \Vert_{\mathbb L_\nu^2( \R)}\,. \label{eq:tame_X-V}
\end{align}

Note that $\mathcal V$ commutes with $(\mathcal V + \sqrt{-1} \lambda)$ (see also equation~\eqref{eq:V-commute:1}).
Replacing functions $ f$ with $\mathcal V^\gamma  f$ in \eqref{eq:tame_X-V},
we get a constant $C_{\alpha} > 0$ such that
\[
\Vert \mathcal X^\alpha \mathcal V^\gamma  f \Vert_{\mathbb L_\nu^2( \R)} \leq \frac{C_{\alpha}}{\vert \lambda \vert} \sum_{k \leq \alpha+1} \Vert \mathcal X^k \mathcal V^\gamma  g \Vert_{\mathbb L_\nu^2( \R)} \,.
\]

Recall that we simplified notation by writing $f = \mathcal A f$, and we have an expression \eqref{eq:norm_f} for converting Sobolev norms in $\Vert \cdot \Vert$ to Sobolev norms in $\mathbb L_2^\nu(\R)$.  
Then using the commutation relations given in \eqref{eq:sl-com},
it follows that for any $s \in \N$, there is a constant $C_s > 0$ such that
\begin{equation}\label{eq:X-V_estimate_tame}
\Vert (I - X^2 - V^2)^{s} f \Vert \leq \frac{C_{s}}{\vert \lambda \vert} \Vert (I - X^2 - V^2)^{(2s+1)/2} g \Vert
\end{equation}
By interpolation, the above estimate holds for any $s \geq 0$, which is the second estimate in Theorem~\ref{main_thm2}.

Next, we consider the full Sobolev norm.  As above, for any $s \in \N$,
there is a constant $C_s > 0$ such that
\[
\begin{aligned}
\Vert f \Vert_{2s} & \leq \sum_{\alpha + \beta + \gamma \leq 2s} \Vert U^\beta X^\alpha V^\gamma f \Vert \\
& \leq C_s \sum_{\alpha + \beta + \gamma \leq 2s} \frac{(1 + \vert \lambda \vert^{-\beta})}{\vert \lambda \vert} \sum_{\substack{j + k \leq \alpha + 2\beta + 1\\ j \leq \beta}} \Vert  U^j X^k V^\gamma g \Vert \\
& \leq C_s \frac{(1 + \vert \lambda \vert^{-2s})}{\vert \lambda \vert} \Vert g \Vert_{4s + 1}\,.
\end{aligned}
\]
The case for $s \geq 0$ now follows by interpolation.
The proof of Theorem~\ref{main_thm2} is now complete.
\end{proof}

\subsection{Proof of Theorem~\ref{main_thm1}}
We adapt the argument given in \cite{flaminio2016effective} to our model using the vector fields $\mathcal U, \mathcal X$ and $\mathcal V$.
Our estimate for Theorem~\ref{main_thm1} is stronger than the one in \cite{flaminio2016effective} (i.e. Theorem~7.2 of \cite{flaminio2016effective}), because it is deduced using our stronger estimate for the twisted equation obtained in Theorem~\ref{main_thm2}.

Sharp, tame Hilbert Sobolev estimates for solutions of the cohomological equation of the horocycle flow was previously obtained by Flaminio and Forni in \cite{flaminio2003invariant}.  In particular, they defined a Green operator $G_\mu^V$ for the horocycle flow on the space of coboundaries in $H_\mu$ for the horocycle flow,
which is uniquely defined up to additive constants on smooth functions by the cohomological equation
\[
V G_\mu^V(g) = g
\]
for the horocycle flow.
Let $G_\mu^{\mathcal V}$ be defined by $G_\mu^{\mathcal V}(\mathcal A g) := \mathcal A G_\mu^V(g) \in \mathbb L_\nu^2(\R)$.
Then
\[
\mathcal V G_\mu^{\mathcal V}(\mathcal A g) = \mathcal A g\,.
\]
Recalling that $\mathcal A$ is unitary, we simplify notation as in the
previous section and write $g = \mathcal A g$.
Because $\mathcal V = -\sqrt{-1} \xi$\,, we get
\begin{equation}\label{eq:flow_cobound}
G_\mu^{\mathcal V}(g) = \frac{\sqrt{-1}}{\xi} g(\xi)\,.
\end{equation} 

In the model $\mathbb L_\nu^2(\R)$, 
the cohomological equation for the horocycle map $h_L$ is
\begin{equation}\label{eq:map_exp_coeqn}
 f(\xi) = \frac{ g(\xi)}{e^{-\sqrt{-1} \xi L} - 1}\,.  
\end{equation}  
We recall from Theorem~\ref{theo:invariant-maps} that 
the space of invariant distributions for the map that are not invariant for the flow 
is $\mathcal I_{L, \text{twist}}(M) \subset W^{-(1/2 + \epsilon)}(M)$ for any $\epsilon > 0$.  
Restricting to the the line or upper half-plane model of $\SL(2, \R)$, 
an explicit spanning set was first described in \cite{tanis2014cohomological}.  
In the Fourier model $\mathbb L_\nu^2(\R)$ and in \cite{flaminio2016effective}, 
this spanning set is given by 
\begin{equation}\label{eq:inv_maps}
\mathcal D_m(g) := g(\frac{2\pi m}{L})\,,  
\end{equation} 
where $m \in \Z$ if $H_\mu$ is in the principal or complementary series, 
and $m \in \N^+$ otherwise.  
Below we use that each distribution $\mathcal D_m$ is of the same form as the distribution \eqref{eq:inv_dist_twist} 
used to study the twisted equation \eqref{eq:twist-cohomological}.  

\begin{proof}[Proof of Theorem~\ref{main_thm1}] 
First let $\alpha, \beta \in \N$ and $g \in Ann_L^{\alpha+2\beta + 1+\epsilon}$, where $L > 0$.  
It will be convenient to estimate 
\begin{equation}\label{eq:coeqn:2}
e^{-\sqrt{-1} \xi L/2}  f(\xi)  = \frac{e^{-\sqrt{-1} \xi L/2}}{e^{-\sqrt{-1} \xi L} - 1}  g(\xi)\,.
\end{equation}
We first focus on $f$ restricted to the interval $I_L$ given by 
\[
I_L := [-\frac{\pi}{L}, \frac{\pi}{L}]\,.  
\]
Because $g$ is also a coboundary for the horocycle flow, from \eqref{eq:flow_cobound}
we can write
\[
\eqref{eq:coeqn:2} = \frac{1}{L}\left(e^{-\sqrt{-1} \xi L/2} \frac{-\sqrt{-1} \xi L}{e^{-\sqrt{-1} \xi L} - 1}\right)  G_\mu^{\mathcal V}( g)(\xi)\,.
\]
The function $\phi(\eta) = e^{-\sqrt{-1} \eta/2} \frac{-\sqrt{-1} \eta}{e^{-\sqrt{-1} \eta} - 1}$ is
smooth on $[-\pi,\pi]$, and $\xi L \in [-\pi, \pi]$ whenever $\xi \in I_L$.
Then by Lemma~\ref{lemm:XU-f:1}, for any $\alpha, \beta \in \N$,
we get coefficients $(b_{j k l m n}^{(\alpha \beta)})$ such that
\[
\begin{aligned}
\mathcal U^\beta & \mathcal X^\alpha e^{-\sqrt{-1} \xi L/2}  f  = \frac{1}{L} \mathcal U^\beta \mathcal X^{\alpha} (\phi(\xi L) G_\mu^{\mathcal V}( g)) \\
& = \frac{i^\beta}{L} \sum_{\substack{j + k + l +n \leq \alpha +  \beta \\ j +l +m \leq \beta}}  b_{j k l m n}^{(\alpha \beta)} \left((\mathcal U^j \mathcal X^{k}  G_\mu^{\mathcal V}( g)) ((\xi \frac{d^2}{d\xi^2})^l (\frac{d}{d\xi})^m \mathcal X^{n} \phi(\xi L))\right)\,.
\end{aligned}
\]
Recall that $\mathcal X = - 2\xi \frac{d}{d\xi}$,
so there is a constant $C_{\alpha, \beta} > 0$ such that
\[
\Vert ((\xi \frac{d^2}{d\xi^2})^l (\frac{d}{d\xi})^m \mathcal X^{n} \phi(\xi L)) \Vert_{L^\infty(I_L)}
 \leq C_{\alpha, \beta} (1 + L^{2l + m}) \,.
\]
So the triangle inequality gives
\begin{align}
\Vert \mathcal U^\beta & \mathcal X^\alpha e^{-\sqrt{-1} \xi L/2}  f \Vert_{\mathbb L_\nu^2(I_L)} \notag \\ & \leq \frac{C_{\alpha, \beta}}{L} \sum_{\substack{j + k + l +n \leq \alpha +  \beta \\ j +l+ m \leq \beta}} (1 + L^{2l+m}) \Vert \mathcal U^j \mathcal X^{k}  G_\mu^{\mathcal V}( g)\Vert_{\mathbb L_\nu^2(I_L)} \notag \\
& \leq C_{\alpha, \beta} \frac{1 + L^{\alpha + 2 \beta}}{L} \sum_{j + k \leq \alpha +  \beta} \Vert \mathcal U^j \mathcal X^{k}  G_\mu^{\mathcal V}( g)\Vert_{\mathbb L_\nu^2(I_L)}\,. \label{eq:flow_terms}
\end{align}

When $\xi \in \R \setminus I_L$, we only consider the case $\xi \in \R^+$, as the case $\xi \in \R^-$ is analogous.
Write
\[
\frac{e^{-\sqrt{-1} \xi L/2}}{e^{-\sqrt{-1} \xi L} - 1}  = \frac{\sqrt{-1}}{2\sin(\xi L/2)}\,,
\]
so
\[
e^{-\sqrt{-1} \xi L/2}  f(\xi)
 = \sqrt{-1} \frac{ g(\xi)}{2\sin(\xi L/2)} \,.
\]
Because
\[
\frac{1}{\sin(\xi L/2)} = \frac{1}{\pi}\left(\frac{2\pi}{\xi L} + \sum_{a \geq 1} (-1)^a \frac{\xi L /\pi}{(\xi L/2\pi)^2 - a^2}\right)\,,
\]
we have
\begin{align}
e^{-\sqrt{-1} \xi L/2}  f(\xi) & = \frac{\sqrt{-1}}{2\pi} \left(\frac{2\pi}{L} \frac{ g(\xi)}{\xi} + \sum_{a \geq 1} (-1)^a \frac{\xi L /\pi}{(\xi L/2\pi)^2 - a^2}  g(\xi)\right)  \notag \\
& = \frac{\sqrt{-1}}{L} \frac{ g(\xi)}{\xi} + \frac{2\sqrt{-1}}{L} \sum_{a \geq 1} (-1)^a \frac{\xi}{\xi^2 - (2\pi a/L)^2}  g(\xi)\,. \notag
\end{align}
By the triangle inequality,
\begin{align}
\Vert \mathcal U^\beta & \mathcal X^\alpha e^{-\sqrt{-1} \xi L/2}  f \Vert_{\mathbb L_\nu^2(\R^+\setminus I_L)} \notag \\
 & \leq \frac{1}{L} \Vert \mathcal U^\beta \mathcal X^\alpha (\frac{ g}{\xi}) \Vert_{\mathbb L_\nu^2(\R^+\setminus I_L)}  + \frac{2}{L}  \sum_{a \geq 1} \Vert \mathcal U^\beta \mathcal X^\alpha \frac{\xi}{\xi^2 - (2\pi a/L)^2}  g\Vert_{\mathbb L_\nu^2(\R^+\setminus I_L)}\,. \label{eq:all_terms2}
\end{align}

We estimate each term.  For the first one, notice
\[
\frac{ g(\xi)}{\xi} = -\sqrt{-1}  G^{\mathcal V}( g)(\xi)\,,
\]
so
\begin{equation}\label{eq:flow_term}
\Vert \mathcal U^\beta \mathcal X^\alpha (\frac{ g}{\xi}) \Vert_{\mathbb L_\nu^2(\R^+\setminus I_L)} \leq \Vert  G^{\mathcal V}( g) \Vert_{\alpha + \beta}\,.
\end{equation}

For the remaining terms, Lemma~\ref{lemm:XU-f:1} gives coefficients $(b_{j k l m n}^{(\alpha \beta)})$
\begin{align}
\mathcal U^\beta & \mathcal X^\alpha \frac{\xi}{\xi^2 - (2\pi a/L)^2}   g(\xi) = \mathcal U^\beta \mathcal X^\alpha \left(\frac{\xi^{1-\epsilon}}{\xi + 2\pi a/L}\right) \left(\frac{\xi^\epsilon  g(\xi)}{\xi - 2\pi a/L}\right) \notag \\
& = \sqrt{-1}^\beta \sum_{\substack{j + k + l + n \leq \alpha + \beta \\ j + l+m \leq \beta}} b_{j k l m n}^{(\alpha \beta)}
(\mathcal U^j \mathcal X^k (\frac{\xi^\epsilon  g(\xi)}{\xi - 2\pi a/L})) \phi_{l, m, n, a}^{(\epsilon)}(\xi)\,, \label{eq:non-flowTerms}
\end{align}
where
\[
\phi_{l, m, n}^{a, (\epsilon)}(\xi) := (\xi \frac{d^2}{d\xi^2})^l (\frac{d}{d\xi})^m \mathcal X^n \left(\frac{\xi^{1-\epsilon}}{\xi + 2\pi a/L}\right)\,.
\]

Notice that for any $\xi \geq 0$ and any $\epsilon > 0$,
\begin{equation}\label{eq:translate_xi}
\xi + \frac{2\pi a}{L} \geq \xi^{1 - \epsilon} (\frac{2\pi a}{L})^\epsilon\,,
\end{equation}
so
\begin{equation}\label{eq:phi_0}
\Vert \phi_{(0, 0, 0)}^{a, (\epsilon)} \Vert_{L^\infty} \leq (\frac{L}{2\pi a})^\epsilon\,.
\end{equation}

Now to estimate the uniform norm of $\phi_{l, m, n}^{(\epsilon)}$ when $(l, m, n) \neq (0, 0, 0)$\,, notice formula~\eqref{eq:decomp_constant_fact} gives constants $(c_{j, k}^{(l, m, n)})$ such that
\[
(\xi \frac{d^2}{d\xi^2})^l (\frac{d}{d\xi})^m \mathcal X^n = \sum_{\substack{j \leq l + n \\ k \leq 2l + m + n \\ k - j \geq l+m}} c_{j, k}^{(l, m, n)} \xi^j (\frac{d}{d\xi})^{k}\,.
\]
A short computation further shows that because $(l, m, n) \neq (0, 0, 0)$, the index $k$ in the above sum satisfies $k \geq 1$.
Now for any $k \in \N \setminus\{0\}$,
\[
\frac{d^k}{d\xi^k} \left(\frac{\xi}{\xi + 2\pi a/L}\right) = \frac{2\pi a}{L} \frac{(-1)^{k+1} k!}{(\xi + 2\pi a/L)^{k+1}}
\]
So
\[
\begin{aligned}
\phi_{l, m, n}^{a, (\epsilon)}(\xi) &= \sum_{\substack{j \leq l + n \\ 1\leq  k \leq 2l + m + n \\ k - j \geq l+m}} c_{j, k}^{(l, m, n)} \xi^j (\frac{d}{d\xi})^{k} \left(\frac{\xi^{1-\epsilon}}{\xi + 2\pi a/L}\right) \\
& =  \sum_{\substack{j \leq l + n \\ 1\leq k \leq 2l + m + n \\ k - j \geq l+m}} c_{j, k}^{(l, m, n)} \xi^j
\left(\frac{\xi^{1-\epsilon-k}}{\xi + 2\pi a/L} \prod_{0 \leq s \leq k-1} (-\epsilon - s) \right. \\
& \left. + \sum_{\substack{t_1 + t_2 = k \\ t_1 \geq 1}} c_{t_1, t_2} \frac{2\pi a}{L} \frac{(-1)^{t_1+1} t_1!}{(\xi + 2\pi a/L)^{t_1+1}} \xi^{-\epsilon - t_2} \prod_{0 \leq s \leq t_2-1} (-\epsilon - s) \right).
\end{aligned}
\]
By the above formula, and by formula~\eqref{eq:translate_xi}, we get
a constant $C_{\alpha, \beta} > 0$ such that
\begin{align}
\Vert \phi_{l, m, n}^{a, (\epsilon)} \Vert_{L^\infty(\R^+ \setminus I_L)} & \leq C_{\alpha, \beta} (1 + L^{l+m}) (\frac{L}{a})^\epsilon \notag \\
& \leq C_{\alpha, \beta} (1 + L^{\beta}) (\frac{L}{a})^\epsilon \,. \label{eq:phi_gen-uniform}
\end{align}

By formulas~\eqref{eq:flow_term}, \eqref{eq:non-flowTerms} and the above estimate, we get
\begin{align}
\sum_{a \geq 1} & \Vert \frac{\xi}{\xi^2 - (2\pi a/L)^2}  g\Vert_{\mathbb L_\nu^2(\R^+\setminus I_L)}  \notag \\
& \leq C_{\alpha, \beta} \sum_{\substack{a \geq 1 \\ j + k + l + n \leq \alpha + \beta \\ j + l+m \leq \beta}}
\Vert (\mathcal U^j \mathcal X^k (\frac{\xi^\epsilon  g}{\xi - 2\pi a/L})) \phi_{l, m, n, a}^{(\epsilon)}\Vert_{\mathbb L_\nu^2(\R^+\setminus I_L)} \notag \\
& \leq C_{\alpha, \beta}
L^{\epsilon}(1 + L^\beta) \sum_{\substack{a \geq 1 \\ j + k \leq \alpha + \beta \\ j \leq \beta}} a^{-\epsilon}
\Vert (\mathcal U^j \mathcal X^k (\frac{\xi^\epsilon  g}{\xi - 2\pi a/L})) \Vert_{\mathbb L_\nu^2(\R^+\setminus I_L)}\,. \label{eq:map-twist:5}
\end{align}

By \eqref{eq:inv_maps}, for any $a \in \N\setminus\{0\}$, we have that 
$g \in Ann_\mu^{2\pi a/L\,, \alpha + 2\beta + 1+\epsilon}$.  
Hence, we use the estimate \eqref{eq:twist_final:1} for the twisted equation, 
and get that for any $j, k \in \N$, 
\[
\begin{aligned}
\Vert (\mathcal U^j \mathcal X^k (\frac{\xi^\epsilon  g}{\xi - 2\pi a/L})) \Vert_{\mathbb L_\nu^2(\R^+\setminus I_L)} &
\leq C_{\alpha, \beta}\frac{L}{a} \sum_{\substack{s + t \leq k + 2j + 1\\ s \leq j}} \Vert \mathcal U^s \mathcal X^t \vert \mathcal V \vert^{\epsilon}  g \Vert_{\mathbb L_\nu^2(\R)} \\
& \leq C_{\alpha, \beta} \frac{L}{a} \Vert g \Vert_{k + 2j + 1+\epsilon} \,.
\end{aligned}
\]

Combining the above estimate with \eqref{eq:map-twist:5}, we get
\[
\begin{aligned}
\eqref{eq:map-twist:5} & \leq \frac{C_{\alpha, \beta}}{L} \left(\Vert G^{V} g \Vert_{\alpha + \beta}
+  L^{1+\epsilon}(1 + L^\beta) \sum_{a \geq 1} a^{-(1+\epsilon)} \Vert g \Vert_{\alpha + 2\beta + 1 + \epsilon} \right) \\
& \leq C_{\alpha, \beta} \left(\frac{1}{L} \Vert G^{V} g \Vert_{\alpha + \beta}  + L^{\epsilon} \frac{(1 + L^\beta)}{\epsilon} \Vert g \Vert_{\alpha + 2\beta + 1 + \epsilon }\right)\,.
\end{aligned}
\]

Then the triangle inequality, the above estimate and \eqref{eq:flow_terms} imply
there is a constant $C_{\alpha, \beta} > 0$ such that
\begin{align}
\Vert \mathcal U^\beta & \mathcal X^\alpha e^{-\sqrt{-1} \xi L/2} f \Vert_{\mathbb L^2_\nu(\R^+)} \notag \\
 & \leq C_{\alpha, \beta} \left(\Vert \mathcal U^\beta \mathcal X^\alpha e^{-\sqrt{-1} \xi L/2} f \Vert_{\mathbb L_\nu^2(I_L)} + \Vert \mathcal U^\beta \mathcal X^\alpha e^{-\sqrt{-1} \xi L/2} f \Vert_{\mathbb L_\nu^2(\R\setminus I_L)}\right) \notag \\
& \leq C_{\alpha, \beta} \left(\frac{1 + L^{\alpha + 2\beta}}{L}\Vert G_\mu^{V}(g)\Vert_{\alpha + \beta} + L^{\epsilon} \frac{(1 + L^\beta)}{\epsilon} \Vert g \Vert_{\alpha + 2\beta + 1 + \epsilon }\right)\,. \label{eq:UX_final:0}
\end{align}

We simplified notation by writing $f = \mathcal A f$.
From the definition of the unitary map $\mathcal A : H_\mu \to \mathbb L^2_{\nu}(\R)$ given in formula~\ref{eq:def_A}, we have that
\[
\mathcal A^{-1} f(\xi) = \mathcal F^{-1} (\xi^{(\nu - 1)/2} f)\,,
\]
where $\mathcal F$ is the Fourier transform.  Then
\[
\begin{aligned}
\Vert \mathcal U^\beta \mathcal X^\alpha e^{-\sqrt{-1} \xi L/2} \mathcal A f \Vert_{\mathbb L^2_\nu(\R)} & = \Vert \mathcal A U^\beta X^\alpha \mathcal A^{-1} (\mathcal A f(\xi) e^{-\sqrt{-1} \xi L/2})  \Vert_{\mathbb L^2_\nu(\R)} \\
& = \Vert \mathcal A U^\beta X^\alpha \mathcal A^{-1} (\hat f(\xi) e^{-\sqrt{-1} \xi L/2} \xi^{-(\nu-1)/2}) \Vert_{\mathbb L^2_\nu(\R)} \\
& = \Vert \mathcal A U^\beta X^\alpha \mathcal F^{-1} (\hat f(\xi) e^{-\sqrt{-1} \xi L/2})  \Vert_{\mathbb L^2_\nu(\R)} \\
& = \Vert \mathcal A U^\beta X^\alpha f(t - \frac{L}{2}) \Vert_{\mathbb L^2_\nu(\R)} \\
& =  \Vert U^\beta X^\alpha f \circ h_{-L/2} \Vert\,.
\end{aligned}
\]
\label{eq:UX_final}

Because the analogous estimate to \eqref{eq:UX_final:0} also holds on $\R^-$,
we conclude using the above equality that
\begin{align}
\Vert U^\beta X^\alpha & f \circ h_{-L/2} \Vert \notag \\
& \leq C_{\alpha, \beta} \left(\frac{1 + L^{\alpha + 2\beta}}{L}\Vert G_\mu^{V}(g)\Vert_{\alpha + \beta} + L^{\epsilon} \frac{(1 + L^\beta)}{\epsilon} \Vert g \Vert_{\alpha + 2\beta + 1 + \epsilon }\right)\,. \notag % \label{eq:UX_final}
\end{align}

Notice that $V$ commutes with the translation operator $h_L$,  so $f \circ h_L - f = g$ implies $(V^\gamma f) \circ h_L - V^\gamma f = V^\gamma g$, for any $\gamma \in \N$.
Then by the commutation relations \eqref{eq:sl-com} and the above estimate,
for any $s \in \N$, there is a constant $C_s > 0$ such that
\begin{align}
\Vert f \circ h_{-L/2} \Vert_s & \leq \sum_{\alpha + \beta + \gamma \leq s} \Vert U^\beta X^\alpha V^\gamma  f \circ h_{L-2} \Vert \notag \\
& \leq C_{s} \sum_{\alpha + \beta + \gamma \leq s} \frac{1 + L^{\alpha + 2\beta}}{L} \Vert V^\gamma G_\mu^{V}(g)\Vert_{\alpha + \beta} \notag \\
& + C_s L^{\epsilon} \frac{(1 + L^\beta)}{\epsilon} \Vert V^\gamma g \Vert_{\alpha + 2\beta + 1 + \epsilon } \label{eq:final_final_map_est} \\
& \leq C_{s} \left(\frac{1 + L^{2s}}{L} \Vert G_\mu^{V}(g)\Vert_{s} + L^{\epsilon} \frac{(1 + L^s)}{\epsilon} \Vert g \Vert_{2s + 1 + \epsilon}\right)\,. \notag
\end{align}
Theorem~\ref{main_thm1} now follows for $s \geq 0$ by interpolation\,.
\end{proof}

\begin{proof}[Proof of Theorem~\ref{eq:tame_X_V}]
This is immediate from \eqref{eq:final_final_map_est}.
%This follows as in \eqref{eq:X-V_estimate_tame} for
%solutions of the twisted equation \eqref{eq:twist-cohomological}.
%Specifically, for $s \in \N$, we expand $(I - X^2 - V^2)^s$ and apply the commutation relations.
%So it is enough to estimate the terms
%\[
%\Vert X^\alpha V^\gamma f \Vert \leq C_{\alpha, \gamma} \Vert X^\alpha V^\gamma f \Vert_{\mathbb L_\nu^2(I_L)} + \Vert X^\alpha V^\gamma f \Vert_{\mathbb L_\nu^2(\R\setminus I_L)}\,,
%\]
%where $\alpha + \gamma \leq 2 s$.
%The estimate of $\Vert X^\alpha V^\gamma f \Vert_{\mathbb L_\nu^2(I_L)}$ follows from \eqref{eq:flow_terms} and the estimate of $\Vert X^\alpha V^\gamma f \Vert_{\mathbb L_\nu^2(\R\setminus I_L)}$ follows from \eqref{eq:UX_final} together with the fact that $V$ commutes with horocycle map $h_L$.
\end{proof}

\subsection{Proof of Theorem~\ref{main_thm3}}
We use the Fourier transform of the line model of the principal series.
Recall the formulas for vector fields $\hat U, \hat V$ and $\hat X$ are given in formula~\eqref{eq:Fourier_vf}.
We first consider Sobolev estimates for the solution to the equation \eqref{eq:twist-cohomological}, so
\begin{equation}\label{eq:coeqn-twist-gen_lambda}
  \hat f(\xi) = -\sqrt{-1} \frac{\hat g(\xi)}{(\xi - \lambda)}\,,
\end{equation}
which reduces to studying
\begin{equation}\label{eq:twist_eqn-lambda=1}
  \hat f(\xi) = -\sqrt{-1} \frac{\hat g(\xi)}{\lambda(\xi - 1)}\,,
\end{equation}
as in \eqref{eq:twist_est}.

Let $\vert \nu \vert \geq 4$, and let
\[
I_\nu := [1, 1+ \frac{1}{\vert \nu \vert}]\,,
\]
and let $\hat q \in C_c^\infty([\frac{3}{4}, \frac{4}{3}])$
be identically one on the interval $I_\nu$.
Let $\hat g \in C_c^\infty([\frac{3}{4}, \frac{4}{3}])$
be defined for any $\xi \in \R$ by
\begin{equation}\label{eq:fdef}
\hat g(\xi) := \hat q(\xi) (\xi^{\nu+1} - 1)\,.
\end{equation}
Notice that $g \in H_\mu^\infty$ and $\hat g(1) = 0$.
Then by Theorem~3.4 of \cite{flaminio2016effective}, the equation \eqref{eq:twist-cohomological} has a solution $f \in H_\mu^\infty$.

The following property will be used several times: for all $m \in \Z$,
\begin{equation}\label{eq:Vxi^nu=0}
 \hat U \xi^{\nu+m} = -\sqrt{-1} \ m(\nu+m) \xi^{\nu+m-1}\,.
\end{equation}

\begin{lemma}\label{f-norm-est}
Let $s \geq 0$.  Then there is a constant $C_s^{(0)} > 0$ such that
\[
\Vert g \Vert_s \leq C_s^{(0)} (1+|\nu|)^s\,.
\]
\end{lemma}
\begin{proof}
First let $s \in 2 \N$.
By the commutation relations \eqref{eq:sl-com},
and by the triangle inequality,
we have a constant $C_s > 0$ such that
\begin{align}
\Vert (I - V^2 - X^2 - U^2)^{s/2} g\Vert_{0} & \leq C_s \sum_{k+ m+ n\leq s}
\Vert V^k X^m U^{n} g \Vert_{0} \notag \\
& = C_s \sum_{k+ m+ n\leq s}
\Vert \hat V^k \hat X^m \hat U^{n} \hat g \Vert_{0}\label{eq:sumXUV-1} \,.
\end{align}

From formula~(39) and Lemma~3.7 of \cite{flaminio2016effective}, we have
Leibniz-type formulas for the operators $\hat X$ and $\hat U$.
Specifically, there are universal coefficients $(a^{(\alpha)}_\ell)$ and
 $(b^{(\beta)}_{ijkm})$ such that for any pair of functions
 $g_1, g_2$, we have
 \begin{equation}\label{eq:Leibniz}
 \begin{aligned}
   & \hat X^\alpha (g_1 g_2)  =  \sum_{\ell=0}^{\alpha} a^{(\alpha)}_\ell \hat X^\ell g_1  (\hat X -(1 - \nu))^{\alpha-\ell} g_2) \\
 &\hat U^\beta(g_1 g_2)  = \sum_{
\substack{
 i+j+ m\leq \beta \\
 k \leq m
  }}
 b^{(\beta)}_{ijkm} [(\frac{d}{d\xi})^m \hat U^i g_1]
 [(\hat X- (1-\nu))^k \hat U^j g_2]\,.
 \end{aligned}
 \end{equation}
Set
\[
g_1(\xi) := \hat q(\xi)\,; \qquad g_2(\xi) := \xi^{\nu+1} - 1\,.
\]
By \eqref{eq:Vxi^nu=0}, and because $\hat X$ is only a first order differential operator, it follows that there is a constant $C_{s, q} > 0$ such that
\[
\Vert \hat X^m \hat U^{n} \hat g \Vert_{0} \leq C_{s, q} (1 + \vert \nu \vert)^{m+n}\,.
\]
Because $\hat g$ is supported on a bounded interval, it follows from the formula for $\hat U$ and the above estimate that
\[
\eqref{eq:sumXUV-1} \leq C_{s, q} (1 + \vert \nu \vert)^s\,.
\]
Because $q$ is fixed, we have now proven the lemma in the
case that $s \in 2 \N$.
The lemma for $s \geq 0$ follows by interpolation.
\end{proof}

Now we focus on a lower bound for the Sobolev norm of
the solution $f$, given by \eqref{eq:twist_eqn-lambda=1}.
Observe that for any $s \geq 0$,
\[
\Vert f \Vert_s \geq \Vert (I - U^2)^{s/2} f \Vert\,.
\]
\begin{proposition}\label{prop:lower_twist}
For every $s \geq 0$,
there are constants $c_s^{(1)} > 4$\,, $C_s^{(1)} > 0$ such that
for all $\vert \nu \vert > c_s^{(1)}$,
\[
\Vert (I - U^2)^{s/2} f \Vert \geq \frac{C_s^{(1)}}{\vert \lambda \vert} |\nu|^{2s+1/2}\,.
\]
\end{proposition}
\begin{proof}
Clearly,
\begin{align}
\Vert (I - U^2)^{s/2} f \Vert & = \Vert (I - \hat U^2)^{s/2} \hat f \Vert_{L^2(\R)}  \notag \\ 
& \geq \Vert (I - \hat U^2)^{s/2} \hat f \Vert_{L^2(I_\nu)}\,. \label{eq:lower_bnd:0}
\end{align} 
We prove that there is a constant $C_s > 0$ such that
\[
\eqref{eq:lower_bnd:0} \geq C_s |\nu|^{2s+1/2}\,,
\]

We begin by considering integer powers of $\hat V$.
Let $\beta \in \N\setminus\{0\}$.
Because $\hat g(1)= 0$, the fundamental theorem of calculus
shows that for all $\xi \in \R$,
\[
\hat f(\xi) = -\frac{\sqrt{-1}}{\lambda} \int_0^1 \hat g'(1 + t( \xi - 1)) dt
\]
(see Lemma~3.5 of \cite{flaminio2016effective}).
Then for any $t \in [0, 1]$,
set
\[
\xi_t:= 1 + t (\xi-1)\,.
\]
Formula (51) of \cite{flaminio2016effective} gives, by a short computation,
that
  \begin{equation}
    \label{eq:VXder}
\hat U^\beta \hat f(\xi) = \frac{(-\sqrt{-1})^{\beta+1}}{\lambda}\int_0^1 t^\beta [ (\hat U +(t-1)  \frac{d^2}{d\xi^2})^\beta \hat g'] (\xi_t) dt \,.
  \end{equation}

Next, we now expand and rewrite the expression $[ \hat U +(t-1)  \frac{d^2}{d\xi^2}]^\beta$.
Let $W_0:= \frac{\partial^2}{\partial \xi^2}$ and $W_1 := \hat U$.
For each $0 \leq m \leq \beta - 1$ and $1 \leq n \leq \beta$,
let $\mathcal S_{m, n}$ be the set of all sequences of length $m+n$
consisting of $m$ $0$'s and $n$ $1$'s.
Then
\[
\begin{aligned}
(\hat U+(t-1)\frac{d^2}{d\xi^2})^\beta = (t-1)^{\beta} (\frac{\partial^2}{\partial \xi^2})^{\beta} +
\sum_{
\substack{
m+n\leq\beta\\
m \leq \beta-1}} (t-1)^m A_{m,n}\,,
\end{aligned}
\]
where,
\[
A_{m, n} := \sum_{(l_i) \in \mathcal S_{m, n}} \prod_{i = j}^{\beta} W_{l_j}\,.
\]

Hence,
\[
\begin{aligned}
\eqref{eq:VXder} & = \frac{(-\sqrt{-1})^{\beta+1}}{\lambda}\int_0^1 t^\beta (t-1)^\beta [(\frac{d^2}{d\xi^2})^\beta \hat g'](\xi_t) dt  \\
& + \frac{(-\sqrt{-1})^{\beta+1}}{\lambda} \sum_{
\substack{
m+n\leq\beta\\
m \leq \beta-1}} \int_0^1 t^\beta (t-1)^m A_{m,n} \hat g(\xi_t) dt \,.
\end{aligned}
\]
Then in the above expression, we show that the first term is relatively large,
and the sum is relatively small.  Proposition~\ref{prop:lower_twist} will follow
by the triangle inequality.

\begin{lemma}\label{lemm:upper0}
For any $\beta \in \N$, there is a constant $C_\beta^{(2)} > 0$ such that
for any $\xi \in I_\nu$, and for any $\vert \nu \vert \geq 4$,
\begin{equation}\label{eq:expand-V-d/dxi-1}
\vert \int_0^1 [t^\beta \sum_{
\substack{
m+n\leq\beta\\
m \leq \beta-1}}
(t-1)^m A_{m,n} \hat g'](\xi_t) dt\vert \leq C_\beta^{(2)} \vert \nu\vert^{2\beta}\,.
\end{equation}
\end{lemma}
\begin{proof}
Notice that for any $k \in \N\setminus\{0\}$,
\[
\frac{d^k}{d\xi^k} \xi^{\nu+ 1} = \prod_{j=0}^{k-1} (\nu+1-j) \xi^{\nu-k+1}\,.
\]
So by \eqref{eq:Vxi^nu=0} and by the conditions on the sum \eqref{eq:expand-V-d/dxi-1}
there is a constant $C_\beta > 0$ such that
for any $\xi \in I_\nu$,
\[
\begin{aligned}
\vert A_{m,n} \hat g'(\xi)\vert\vert & \leq C_\beta \vert \nu \vert^{2m+n+1} \\
& \leq C_\beta \vert \nu \vert^{2\beta}\,.
\end{aligned}
\]
Hence, there is a constant $C_\beta > 0$ such that
for all $\xi \in I_\nu$ and $t \in [0, 1]$,
\[
\begin{aligned}
\vert \sum_{
\substack{
m+n\leq\beta\\
m \leq \beta-1}} (t-1)^m A_{m,n} \hat g'(\xi)\vert \leq C_\beta  \vert \nu\vert^{2\beta} \,.
\end{aligned}
\]
\end{proof}

It remains to prove the following lower bound.

\begin{lemma}\label{prop:int-omega}
For all $\beta \in \N \setminus\{0\}$, there is a constant $C_\beta^{(3)} > 4$ such that for all $\vert \nu \vert \geq C_\beta^{(3)}$ and for all $\xi \in I_\nu$, we have
\[
\vert \int_0^1 t^\beta (t-1)^{\beta} [(\frac{\partial^2}{\partial \xi^2})^{\beta} \hat g'](\xi_t) dt \vert  \geq \frac{1}{C_\beta^{(3)}} \vert \nu \vert^{2\beta+1}\,.
\]
\end{lemma}
\begin{proof}
We have
\[
(\frac{\partial^2}{\partial \xi^2})^{\beta} \hat g'(\xi) = \xi^{\nu - 2\beta} \prod_{l = 0}^{2\beta} (\nu+1-l)\,,
\]
Notice that $\nu \in \sqrt{-1} \R$ and $\vert \nu \vert \geq 4$.
Then by assumption on $\xi$, we have
\[
\begin{aligned}
\vert (\frac{\partial^2}{\partial \xi^2})^{\beta} \hat g'(\xi)\vert & = \vert (\nu+1) (\frac{\partial^2}{\partial \xi^2})^{\beta} \xi^{\nu} \\
& = \vert \xi^{-2\beta} \prod_{l = 0}^{2\beta} (\nu+1-l)\vert \\
& \geq \vert \nu \vert^{2\beta+1} \vert \xi \vert^{-2\beta}\,.
\end{aligned}
\]
So
\begin{equation}\label{eq:lower(3)}
\vert \int_0^1 t^\beta (t-1)^{\beta} (\frac{\partial^2}{\partial \xi^2})^{\beta} \hat g'(\xi_t) dt \vert \geq \vert \nu \vert^{2\beta+1} \vert \int_0^1 t^\beta (t-1)^{\beta} \xi_t^{\nu - 2\beta} dt \vert \,.
\end{equation}

Next, we can write
\[
\begin{aligned}
\xi_t^\nu & = \exp(\sqrt{-1} \text{sgn}(\nu) \vert \nu\vert \log (1+t(\xi-1))) \\
& = \cos(\text{sgn}(\nu) \vert\nu\vert \log (1+t(\xi-1))) + \sqrt{-1} \sin(\text{sgn}(\nu) \vert\nu\vert \log (1+t(\xi-1)))\,.
\end{aligned}
\]
So
\[
\begin{aligned}
\eqref{eq:lower(3)} & \geq \vert \nu \vert^{2\beta+1} \vert  \int_0^1 t^\beta (t-1)^{\beta} \xi_t^{\nu-2\beta} dt \vert \notag \\
&  \geq \vert \nu \vert^{2\beta+1} \vert \re \int_0^1 t^\beta (t-1)^{\beta} \xi_t^{\nu-2\beta} dt \vert \notag  \\
& \geq \vert \nu \vert^{2\beta+1} \vert \int_0^1 t^\beta (t-1)^{\beta} \xi_t^{-2\beta}\cos( \vert \nu\vert \log (1+t(\xi-1))) dt \vert \notag \\
& = \vert \nu \vert^{2\beta+1} \vert \int_0^1 t^\beta (1-t)^{\beta} \xi_t^{-2\beta} \cos( \vert \nu\vert (t (\xi-1) + \phi(t, \xi)))dt \vert \,, \label{eq:lower(1)}
\end{aligned}
\]
where for any $t \in (0, 1)$,
$\phi$ is given by 
\[
\phi(t, \xi) := \log(1+t(\xi-1)) - t(\xi-1)\,.  
\]
Then for $\vert \nu \vert \geq 4$,
\[
\begin{aligned}
\vert \phi(t, \xi) \vert & \leq \frac{1}{2} (t (\xi - 1))^2 (1 + \frac{2}{3} t(\xi-1)) \\
& \leq \frac{1}{\vert \nu \vert^2}\,.
\end{aligned}
\]
Then
\begin{align}
\cos( \vert \nu\vert \log (1+t(\xi -1) &)= \cos\left(\vert\nu\vert (t (\xi-1) + \phi(t, \xi))\right)  \notag \\
& \geq \cos(1+ \frac{1}{\vert \nu \vert}) > \frac{1}{4}\,.\label{eq:cos>0}
\end{align}
So the integrand in \eqref{eq:lower(1)} is positive
for any $t \in (0, 1)$, and in particular,
\begin{align}
\eqref{eq:lower(1)} & = \vert \nu \vert^{2\beta+1} \int_0^1 t^\beta (1-t)^{\beta} \xi_t^{-2\beta} \cos( \vert\nu\vert (t (\xi-1) + \phi(t, \xi)))dt \notag \\
& > \frac{\vert \nu \vert^{2\beta+1}}{4} \int_0^1 t^\beta (1-t)^{\beta} \xi_t^{-2\beta} dt\,. \label{eq:lower_2}
\end{align}

Next, because $\xi \in I_\nu$ and $\vert \nu \vert \geq 4$, we have
\[
 \xi_t^{-2\beta} \geq (1 + \frac{1}{\vert \nu \vert})^{-2\beta} > 4^{-\beta} \,,
\]
which means
\begin{equation}\label{eq:lower:final-1}
\eqref{eq:lower_2} \geq 4^{-(\beta+1)} \vert \nu \vert^{2\beta+1} \int_0^1 t^\beta (1-t)^{\beta} dt\,.
\end{equation}
Finally, restricting to the interval $t \in [\frac{1}{4}, \frac{3}{4}]$, we get
\[
\eqref{eq:lower:final-1} > 4^{-(3\beta +2)} \vert \nu \vert^{2\beta+1}  \,.
\]
\end{proof}

We now prove Proposition~\ref{prop:lower_twist}.
Because the parameter $\nu$ can be arbitrarily large
in absolute value, let $\nu$ satisfy
\[
\left\{
\begin{aligned}
& \vert \nu \vert \geq C_\beta^{(3)}\,,  \\
& \frac{1}{C_\beta^{(3)}}- \frac{C_\beta^{(2)}}{\vert \nu \vert} \geq \frac{1}{2 C_\beta^{(3)}} \,.
\end{aligned}
\right.
\]
Then by formula~\ref{eq:expand-V-d/dxi-1}, by the triangle inequality and by Lemmas~\ref{lemm:upper0} and \ref{prop:int-omega}, we get
\[
\begin{aligned}
\Vert \hat U^\beta \hat f \Vert_{L^2(I_\nu)} & \geq \frac{1}{\vert\lambda\vert} \left|\Vert \int_0^1 t^\beta (t-1)^{\beta} [(\frac{\partial^2}{\partial \xi^2})^{\beta} \hat g'](\xi_t) dt \Vert_{L^2(I_\nu)}\right. \\
&-  \left.\Vert \int_0^1 [t^\beta \sum_{
\substack{
m+n\leq\beta\\
n \geq 1}}
(t-1)^m A_{m,n} \hat g'](\xi_t) dt\Vert_{L^2(I_\nu)}\right|  \\
& \geq \frac{\vert \nu \vert^{2\beta+1/2}}{\vert\lambda\vert} \left(\frac{1}{C_\beta^{(3)}}- \frac{C_\beta^{(1)}}{\vert \nu \vert}\right) \\
& \geq \frac{1}{2 C_\beta^{(3)} \vert \lambda \vert}  \vert \nu \vert^{2\beta+1/2}\,.
\end{aligned}
\]
Hence, there is a constant $c_\beta^{(1)} > 4$ such that
for any $\vert \nu \vert \geq c_\beta^{(1)}$, we have
\[
\Vert (I - U^2)^{\beta/2} f \Vert > \frac{C_\beta^{(1)}}{\vert \lambda \vert} \vert \nu \vert^{2\beta+1/2}\,.
\]
Moreover, we may take $(c_\beta^{(1)})_{\beta \in \N}$ to be an increasing sequence
in $\beta$.

Now let $s \in \R^+\setminus\N$, and define $\beta = \lfloor s \rfloor$.
Furthermore, let $c_s^{(1)} := c_{\beta + 1}^{(1)} > c_{\beta}^{(1)}$.
Then for any $\vert \nu \vert \geq c_s^{(1)}$, we have
\[
\begin{aligned}
& \Vert (I - U^2)^{\beta/2} f \Vert > \frac{C_\beta^{(1)}}{\vert \lambda \vert} \vert \nu \vert^{2\beta+1/2}\,, \\
& \Vert (I - U^2)^{(\beta+1)/2} f \Vert > \frac{C_{\beta+1}^{(1)}}{\vert \lambda \vert} \vert \nu \vert^{2(\beta+1)+1/2}\,.
\end{aligned}
\]
Then the estimate for $s$ follows by interpolation.
This concludes the proof of Proposition~\ref{prop:lower_twist}.
\end{proof}

\begin{proof}[Proof of Theorem~\ref{main_thm1}]
Let $C > 0$, and let $s \in \N \setminus\{0\}$,
$\lambda \in \R^*$ and let $\sigma \in [0, s+1/2)$.
Recall that $\hat g_\lambda(\xi) := \hat g(\lambda \xi)$,
so from \eqref{eq:fdef}, we take our example to be
\begin{equation}\label{eq:f-lambda}
\hat g (\xi) = \hat q(\lambda^{-1} \xi) (\lambda^{-(\nu+1)} \xi^{\nu+1} - 1)\,,
\end{equation}
which satisfies $\hat g(\lambda) = 0$.
Let the positive constants $C_s^{(0)}$
and $C_s^{(1)}$ be from Lemma~\ref{f-norm-est}
and Proposition~\ref{prop:lower_twist}, respectively.
Because the parameters $\nu$ for the principal series
can be arbitrarily large in absolute value,
take $\vert \nu \vert$ large enough so that
Proposition~\ref{prop:lower_twist} holds and
\begin{equation}\label{eq:nu-large-enough}
 C_s^{(1)} \vert \nu \vert^{2s+1/2}  > C C_s^{(0)} (\vert \lambda \vert^{-(2s+\sigma)} + \vert \lambda \vert^{\sigma} ) \vert \nu \vert^{s+\sigma}\,.
\end{equation}

Let $\hat \triangle$ be the Fourier transform of the operator $\triangle$.
As in (56) of \cite{flaminio2016effective}, notice that for any $a, b, c \in \N$,
\[
\begin{aligned}
\hat V^a \hat X^b \hat U^c \hat g_\lambda = \lambda^{c-a} (\hat V^a \hat X^b \hat U^c \hat g)_\lambda\,.
\end{aligned}
\]
Therefore,
\[
\begin{aligned}
\Vert (1 + \hat \triangle)^{(s+\sigma)/2} \hat g_\lambda \Vert  & = \Vert [(I - \lambda^{-2} \hat V^2 - \hat X^2 - \lambda^{2} \hat U^2)^{(s+\sigma)/2} \hat g]_\lambda \Vert  \\
& \geq \min\{\vert \lambda\vert^{-(s+\sigma)}, \vert \lambda \vert^{(s+\sigma)}\} \Vert g \Vert_{s+\sigma} \\
& \geq (\vert \lambda\vert^{-(s+\sigma)} +  \vert \lambda \vert^{s+\sigma})^{-1} \Vert g \Vert_{s+\sigma}\,.
\end{aligned}
\]

Then by \eqref{eq:nu-large-enough},
\[
\begin{aligned}
\Vert (I - \hat \triangle)^{s/2} \hat f)_{\lambda} & \Vert
\geq \Vert (\hat U^s \hat f)_{\lambda} \Vert  \\
&  = \vert \lambda \vert^{-s} \Vert \hat U^s \hat f_\lambda \Vert  \\
& \geq \vert \lambda \vert^{-s} C_s^{(1)} \vert \nu \vert^{2s-1/2} \\
& > C \vert \lambda \vert^{-s} C_s^{(0)} \vert \nu \vert^{s+\sigma}  (\vert \lambda \vert^{-(2s+\sigma)} + \vert \lambda \vert^{\sigma} ) \\
& \geq C \vert \lambda \vert^{-s}  \Vert \hat g _\lambda\Vert_{s +\sigma} (\vert \lambda \vert^{-(2s+\sigma)} + \vert \lambda \vert^{\sigma} ) \\
& \geq C \frac{\vert \lambda \vert^{-s}}{\vert \lambda\vert^{-(s+\sigma)} +  \vert \lambda \vert^{s+\sigma}} \Vert ((I - \hat \triangle)^{(s+\sigma)/2} g )_\lambda\Vert_{0} (\vert \lambda \vert^{-(2s+\sigma)} + \vert \lambda \vert^{\sigma} )  \\
& \geq C  \Vert ((I - \triangle)^{(s+\sigma)/2} g )_\lambda\Vert_{0} \,.
\end{aligned}
\]
Therefore,
\[
\Vert f \Vert_s > C \Vert g \Vert_{s+\sigma}\,.
\]

The estimate for $s \geq 0$ follows by interpolation,
which completes the proof of the theorem.
\end{proof}

\begin{proof}[Proof of Theorem~\ref{main_thm3}]
Using the Fourier transform in the line model, the
cohomological equation \eqref{eq:co_eqn-map}
for unipotent maps has the form
\begin{equation}\label{eq:Fourier-coeqn}
(e^{-L \sqrt{-1} \xi} - 1)\hat f(\xi) = \hat g(\xi)\,.
\end{equation}
Define
\[
\begin{aligned}
\hat g^{twist}(\xi) := \hat q(\frac{L}{2\pi}\xi) [(\frac{L}{2\pi}\xi)^\nu - 1]\,,
\end{aligned}
\]
and
\[
\hat f(\xi) := \frac{\hat g^{twist}(\xi)}{\frac{L}{2\pi}\xi - 1)}\,.
\]
So for $\lambda = \frac{2\pi}{L}$, $\hat g^{twist}$ is given by
$\hat g_\lambda$ in formula \eqref{eq:fdef},
and the above $\hat f$ is given by $-\sqrt{-1} \lambda^{-1}$ times the function
$\hat f$ from formula \eqref{eq:coeqn-twist-gen_lambda}.

Further define $H$ on $\R$ by
\[
H = \left(\frac{e^{-L \sqrt{-1} \xi} - 1}{\frac{L}{2\pi} \xi - 1}\right)\,,
\]
and notice that
\begin{equation}\label{eq:H-smooth}
H, H^{-1} \in C^\infty(\frac{2\pi}{L} [\frac{3}{4}, \frac{4}{3}])\,.
\end{equation}
Define
\[
\begin{aligned}
\hat g & = g^{twist} \cdot H\,,
\end{aligned}
\]
so
\[
\begin{aligned}
(e^{-L \sqrt{-1} \xi} - 1)\hat f(\xi) & = (e^{-L \sqrt{-1} \xi} - 1) \frac{\hat g^{twist}(\xi)}{(\frac{L}{2\pi}\xi - 1)} \\
& = H (\xi) \hat g^{twist}(\xi) \\
& = g(\xi)\,,
\end{aligned}
\]
where because $\hat q$ is supported on $[ \frac{3}{4}, \frac{4}{3}]$,
we get that $\hat f$ and $\hat g$ are supported on $\frac{2\pi}{L}[ \frac{3}{4}, \frac{4}{3}]\,.$

Then for any $s \in \N$, the commutation relations give a constant $C_s > 0$ such that
\begin{align}
\Vert g \Vert_{s} & \leq C_{s} \sum_{m + n + \beta \leq s}
\Vert V^m X^n U^\beta g\Vert  \notag \\
& \leq C_{s} \sum_{m + n + \beta \leq s}
\Vert \hat V^m \hat X^n \hat U^\beta (H \cdot \hat g^{twist}) \Vert
 \label{eq:g-map:22}
\end{align}
The Leibniz-type formula for $\hat V$ (see \eqref{eq:Leibniz}) gives
universal coefficients $ (b^{(\beta_2)}_{ijkm})$ such that
\[
 \hat U^{\beta}(\hat g^{twist} (\frac{L}{2\pi}\xi) \cdot H ) = \sum_{\substack { i+j+ m\leq
  \beta \\
 k \leq m }}
 b^{(\beta)}_{ijkm} [(\frac{d}{d\xi})^m \hat U^i H(\xi) ]
 [(\hat X - (1-\nu))^k \hat U^j \hat g^{twist}( \frac{L}{2\pi}\xi)]\,.
\]
Define $H_L$ by $H(\xi) = H_L(L \xi)$.
Then because $H_L$ is smooth and independent of $\nu$ and $L$,
for each $m+i \leq \beta$, we have a constant $C_{\beta}> 0$ such that
\[
\vert (\frac{d}{d\xi})^m \hat U^i H_L(\xi)\vert \leq C_{\beta} \vert \nu \vert^i\,.
\]
Therefore,
\[
\vert (\frac{d}{d\xi})^m \hat U^i H(\xi)\vert \leq C_{\beta} (1 + L^{m + 2i}) \vert \nu \vert^i\,.
\]
Moreover, a calculation shows that for any constant $c$ and for each $j \in \N$,
\begin{equation}\label{eq:scaling}
\hat U^j \hat g^{twist} (c\xi) = c^j [\hat U^j \hat g^{twist}](c \xi)\,, \quad \hat X^k \hat g^{twist} (c \xi) = [\hat X^k \hat g^{twist}](c \xi)\,.
\end{equation}
Then by Lemma~\ref{f-norm-est},
there is a constant $C_\beta > 0$ such that
\[
\begin{aligned}
\Vert \hat U^{\beta}(\hat g^{twist}(\frac{L}{2\pi}\xi)\cdot H)\Vert
& \leq C_{\beta} \sum_{i+j+ m\leq
  \beta}  (1 + L^{m + 2i + j}) \vert \nu \vert^{i + j} \\
& \leq C_{\beta} (1 + L^{2\beta}) \vert \nu \vert^{\beta}\,.
\end{aligned}
\]
Finally, because $H \cdot \hat g^{twist}$ is compactly supported and the
derivatives $\hat U$ and $\hat X$ contribute at most one power of $\nu$,
we get a constant $C_\beta > 0$ such that
\[
\begin{aligned}
\eqref{eq:g-map:22} & \leq C_s (1+L^{2s}) \sum_{m + n + \beta \leq s} \vert \nu \vert^{n + \beta} \\
&  \leq C_s (1+L^{2s}) \vert \nu \vert^{s}\,.
\end{aligned}
\]
It follows by interpolation that for any $s \geq 0$,
there is a constant $C_s^{(2)} > 0$ such that
\[
\Vert g \Vert_s \leq C_{s}^{(2)} (1+L^{2s}) \vert \nu \vert^{s}\,.
\]

On the other hand, by \eqref{eq:scaling} and by Proposition~\ref{prop:lower_twist},
for any $s \geq 0$, there are constants $c_s^{(1)} > 4$, $C_s^{(1)} > 0$ such that
for any $\vert \nu \vert \geq c_s^{(1)}$,
\[
\begin{aligned}
\Vert (I - \hat U^2)^{s/2} \hat f \Vert & =  (\frac{2\pi}{L})^{-1/2} \Vert (I - (\frac{2\pi}{L})^2 \hat U^2)^{s/2} \hat f_{(2\pi/L)} \Vert \\
& > c_s^{(0)} (\frac{L}{2\pi} + \frac{2\pi}{L})^{-(s +1/2)} \Vert (I - \hat U^2)^{s/2} \hat f_{(2\pi/L)}\Vert \\
& > c_s^{(0)} (\frac{L^2 + 4\pi^2}{2\pi L})^{-(s +1/2)} \vert \nu \vert^{2s+1/2}\,.
\end{aligned}
\]

So let $\sigma \in [0, s + 1/2)$  and $C > 0$.
Then for any $\vert \nu\vert$ large enough that
\[
\left\{
\begin{aligned}
& \vert \nu \vert \geq \alpha_{s}^{(0)}\,,  \\
&  c_s^{(0)} (\frac{L^2 + 4\pi^2}{2\pi L})^{-(s +1/2)} \vert \nu \vert^{2s+1/2}  > C C_{s+\sigma, n}^{(2)} (1+L^{2(s+\sigma)}) \vert \nu \vert^{s+\sigma}\,,
\end{aligned}
\right.
\]
we get
\[
\begin{aligned}
\Vert (I - U^2)^{s/2} f \Vert & \geq c_s^{(0)} (\frac{L^2 + 4\pi^2}{2\pi L})^{-(s +1/2)} \vert \nu \vert^{2s+1/2} \\
& \geq C C_{s+\sigma, n}^{(0)} (1+L^{2(s+\sigma)}) \vert \nu \vert^{s+\sigma} \\
& > C \Vert g \Vert_{s+\sigma}\,.
\end{aligned}
\]
\end{proof}

\section{Proof of Theorems~\ref{th:8}}

Recall the definition of the translation operators $h^{(1)}_{L_1}$ and $h^{(2)}_{L_2}$ in \eqref{eq:cocycle_maps}.
Theorem 1.1 of \cite{damjanovic2014cocycle} shows that there is a solution $p\in W^\infty(\mathcal{H})$
such that
\begin{align}
  p\circ h_{L_1}^{(1)}-p=g\quad \text{and}\quad p\circ h_{L_2}^{(2)}-p=f
\end{align}
with estimates $\norm{p}_s\leq C_{s,L_2}\norm{f}_{3s+6}$, $s\geq0$. In \cite{damjanovic2014cocycle} $\mathcal{H}=L^2_0(G/\Gamma)$ and $\pi$ is the regular representation, where $\Gamma$ is an irreducible lattice in $G$ and $L^2_0(G/\Gamma)$ is the space of square integrable functions on $G/\Gamma$ with zero average. Note that the result can be extended to unitary representation of $G$ such that
such that the restriction of $\pi$ to any $\SL(2,\RR)$
factor has a spectral gap. We have the decomposition:
\begin{align*}
 L^2_0(G/\Gamma)=\int_{\otimes_{\mu,\theta}}\mathcal{H}_\mu\otimes\mathcal{H}_\theta
\end{align*}
where $\mathcal{H}_\mu$ and $\mathcal{H}_\theta$ range over all non-trivial irreducible representations of $\SL(2, \R)$. The solution $p$ was constructed, as well as
its estimates were obtained in
\cite{damjanovic2014cocycle} in each $\mathcal{H}_\mu\otimes\mathcal{H}_\theta$. Then discussion in Section \ref{sec:3} shows that $p$ is a bona fide solution.

Hence, it suffices to obtain the tame estimates of $p$ with respect to $f$ and $g$. We use $X_1$, $V_1$ and $U_1$ to denote the basis Lie algebra for the first copy of $\SL(2, \R)$ and $X_2$, $V_2$ and $U_2$ for the second copy.
For $Z \in \{X_2, V_2, U_2\}$, we note that
\begin{align*}
  (Z^np)\circ h_{L_1}^{(1)}-Z^np=Z^ng, \qquad \forall n\in\NN,
\end{align*}
Since the restriction of $\pi$ on the first copy of $\SL(2, \R)$ is still a unitary representation with spectral gap,
by using \eqref{for:6}, it follows that
\begin{align}\label{for:7}
  \Vert Z^np\Vert\leq C(L+\frac{1}{L})\Vert Z^ng \Vert_{2}\leq C(L+\frac{1}{L})\Vert g \Vert_{n+2},\qquad \forall n\in\NN.
\end{align}
Similarly, for $Y \in \{X_1, V_1, U_1\}$, we have
\begin{align}\label{for:8}
  \Vert Y^np\Vert\leq C(L+\frac{1}{L})\Vert Y^nf \Vert_{2}\leq C(L+\frac{1}{L})\Vert f \Vert_{n+2},\qquad \forall n\in\NN,
\end{align}
Then \eqref{for:5} follows directly from \eqref{for:7}, \eqref{for:8} and the following
 elliptic regularity theorem (see \cite[Chapter I, Corollary 6.5 and 6.6]{ter1996elliptic}):
\begin{theorem} Let $\pi$ be a unitary representation of a Lie group $G$ with Lie algebra $\mathfrak{g}$ on a
Hilbert space $\mathcal{H}$. Fix a basis $\{Y_j\}$ for $\mathfrak{g}$ and set $L_{2m}=\sum Y_j^{2m}$, $m\in\NN$. Then
\begin{align*}
    \norm{v}_{2m}\leq C_m(\norm{L_{2m}v}+\norm{v}),\qquad \forall\, m\in\NN
\end{align*}
where $C_m$ is a constant only dependent on $m$ and $\{Y_j\}$.
\end{theorem}

\appendix
\section{Unitary representations of $\SL(2, \R)$}\label{sect:SL2R_reps}
Section~\ref{sec:3} already discussed the direct integral decomposition for unitary representations of general type $I$ Lie groups, but we can say more in the special case of $\SL(2, \R)$.

Recall that $\sl(2, \R)$ is generated by the vector fields
 \[
 \begin{aligned}
 X= \begin{pmatrix} {1}&0\\0& {-1}
 \end{pmatrix}, \quad U=\begin{pmatrix} 0&1\\0& 0
 \end{pmatrix}, \quad V=\begin{pmatrix} 0&0\\1& 0
 \end{pmatrix}.
 \end{aligned}
 \]
 The \emph{Casimir} operator is given by
\begin{align*}
\Box:= -X^2-2(UV+VU)
\end{align*}
and generates the center of the enveloping algebra of $\mathfrak{sl}(2,\RR)$.
Any unitary representation $(\pi, \mathcal{H})$ of $SL(2,\RR)$ is decomposed into a direct integral (see \cite{flaminio2003invariant} and \cite{mautner1950unitary})
\begin{align}\label{for:1}
\mathcal{H}=\int_{\oplus}\mathcal{K}_\mu ds(\mu)
\end{align}
with respect to a positive Stieltjes measure $ds(\mu)$ over the spectrum $\sigma(\Box)$. The
Casimir operator acts as the constant $\mu\in \sigma(\Box)$ on every Hilbert representation space $\mathcal{K}_\mu$, which does not need to be irreducible.
In fact, $\mathcal{K}_\mu$ is in general
the direct sum of an (at most countable) number of unitarily equivalent representation spaces $H_\mu$ equal to the spectral multiplicity of $\mu\in \sigma(\Box)$.   We say that \emph{$\pi$ has a spectral gap} if there is some $u_0>0$ such that $s((0,u_0])=0$. It is clear that if $\pi$ has a spectral gap then $\pi$ contains no non-trivial $SL(2,\RR)$-fixed vectors.

The representation spaces $H_\mu$ have unitarily equivalent models, which we also write as $ H_\mu$, and each is one of the following four classes:
\begin{itemize}
  \item
    If $ \mu \in (0, 1) $, then $H_\mu $ is in the complementary
    series.
  \item
    If $ \mu > 1 $, then $H_\mu $ is in the principal series.
  \item
    If $ \mu = 1 $, then $H_\mu $ is in the mock discrete series or
    the principal series.
  \item
    Otherwise if $ \mu \leq 0 $, then $H_\mu $ is in the discrete series.
\end{itemize}

It will be convenient to describe these representations via models that use a representation parameter $\nu$ satisfying
\[
\mu = 1 - \nu^2\,,
\]
where it is sufficient to take $\nu = \sqrt{1 - \mu}$.
Below are standard models from the literature.  

The line model of the principal or complementary series consists of functions on $\R$ 
and has the following norm.  For $\mu \geq 1$, $\Vert f \Vert_{H_\mu} := \Vert f \Vert_{L^2(\R)}$, 
and when $\mu > 1$, 
\[
  \| f \|_{H_{\mu}} := \left(\int_{\mathbb{R}^2} \frac{f(x) \overline{f(y)}}
  {|x - y|^{1 - \nu}} dx dy\right)^{1/2}\,.
\]
The group action on $H_\mu$ is given by 
\[
A \cdot f(x) := |-b x + d|^{-(\nu + 1)} f(\frac{a x - c}{-b x +
  d})\,,
\]
where $x \in \R$ and 
\begin{equation}\label{eq:A_SL2R}
A = \left(\begin{array}{rr}
    a & b\\
    c & d
  \end{array}\right) \in \SL(2, \mathbb{R})\,.
  \end{equation}  
  This yields the vector fields 
  \[
\begin{array}{cc}
V = -\frac{d}{d x}\,,  \quad \quad  U = (1 + \nu) x + x^{2}\frac{d}{d x}\,, \quad \quad X = (1 + \nu) + 2 x \frac{d}{d x}\,.
\end{array}
\]

The holomorphic discrete or mock discrete series consists of holomorphic functions on the upper half-plane, $\mathbb H$.  Its counterpart, the anti-holomorphic discrete or mock discrete series, does not need to be considered because of a complex anti-linear isomorphism between the two spaces.  
The holomorphic discrete or mock discrete series norm is given by  
\[
\|f\|_{H_\mu} =
\begin{cases}
 \left(\int_{\mathbb H} |f(x+iy)|^2 \,y^{\nu-1}\,dx\,dy\right)^{1/2}  & \quad \quad \nu\ge 1\,,\\
  \left(\sup_{y>0} \int_{\R} |f(x+iy)|^2 \,dx\right)^{1/2} & \quad \quad \nu = 0\,.
\end{cases}
\]

The group action is analogous to the one for the line model.  
For $z \in \mathbb H$ and $A$ given by \eqref{eq:A_SL2R}, 
we have 
\[
A\cdot f(z) :=  |-b z + d|^{-(\nu + 1)} f(\frac{a z - c}{-b z +
  d}),\,,
\]
which yields the vector fields
\[
\begin{array}{ll}
V = -\frac{d}{d z}\,, \quad \quad U = (1 + \nu) z + z^{2}\frac{d}{d z}\,, \quad \quad X = (1 + \nu) + 2 z \frac{d}{d z}\,.
\end{array}
\]

\bibliography{mathBib}{}
\bibliographystyle{plain}

\end{document}